\documentclass[10pt]{article}
\usepackage{lmodern}
\usepackage{amsmath}
\usepackage[T1]{fontenc}
\usepackage[utf8]{inputenc}
\usepackage{authblk}
\usepackage{amsfonts}
\usepackage{graphicx}
\usepackage{tkz-euclide}
\tikzset{elegant/.style={smooth,thick,samples=50,cyan}}
\tikzset{eaxis/.style={->,>=stealth}}
\usepackage{color,tikz}
\usetikzlibrary{decorations.pathreplacing,calc}
\usepackage{rotating}
\tikzset{liltext/.style={font=\tiny}}
\usepackage{amssymb}
\usepackage[english]{babel}
\usepackage{xcolor}
\usepackage{amsthm}
\usepackage{graphicx}
\usepackage{mathrsfs}
\usepackage{makecell}
\usepackage{microtype}
\usepackage{mathscinet}
\usepackage{array}
\usepackage{multirow}
\usepackage{enumerate}
\usepackage[cal=boondoxo,bb=ams]{mathalfa}
\usepackage{hyperref}
\usepackage{booktabs}
\hypersetup{hidelinks}

\newtheorem{theorem}{Theorem}[section]
\newtheorem{prop}{Proposition}[section]

\newtheorem{remark}{Remark}[section]

\newcommand{\ml}{\mathcal}
\newcommand{\mb}{\mathbb}

\DeclareMathOperator{\divv}{div}

\def\XXint#1#2#3{{\setbox0=\hbox{$#1{#2#3}{\int}$ }
		\vcenter{\hbox{$#2#3$ }}\kern-.6\wd0}}

\title{Large time behavior for the classical wave equation with different regular data and its applications}
\author[1]{Wenhui Chen\thanks{Wenhui Chen (wenhui.chen.math@gmail.com)}}
\affil[1]{School of Mathematics and Information Science, Guangzhou University,\authorcr 510006 Guangzhou, China}

\author[2]{Ryo Ikehata\thanks{Ryo Ikehata (ikehatar@hiroshima-u.ac.jp)}}
\affil[2]{Department of Mathematics, Division of Educational Sciences, Graduate School of Humanities and Social Sciences, Hiroshima University, 739-8524 Higashi-Hiroshima, Japan}
\date{}

\begin{document}
		\maketitle

		\begin{abstract}
		\medskip	
In this	paper, we mainly consider large time behavior for the classical free wave equation $u_{tt}-\Delta u=0$ in $\mb{R}^n$.
We derive some large time optimal estimates for the quantity of solution $\|u(t,\cdot)\|_{L^2}$ with initial data belonging to $L^2$ or with additional weighted $L^1$ integrabilities. Particularly, some thresholds are discovered for the (local or global in time) stabilization of this quantity.
 We also apply these results to the wave equation with scale-invariant terms, the undamped $\sigma$-evolution equation, the critical Moore-Gibson-Thompson equation, and the linearized compressible Euler system.\\
			
			\noindent\textbf{Keywords:} wave equation, blow-up in infinite time, large time optimal estimate, stability threshold, regularity of initial data \\
			
			\noindent\textbf{AMS Classification (2020)}  35L05, 35B40, 35B44, 35B35
		\end{abstract}
\fontsize{12}{15}
\selectfont
\section{Introduction}\label{Section_Introduction}\setcounter{equation}{0}
\hspace{5mm}In this manuscript, we are mainly concerned with the well-known Cauchy problem for the following classical wave equation, which is widely applied in elasticity, acoustics and electromagnetism starting in the 16-th century:
\begin{align}\label{Eq-Waves}
\begin{cases}
u_{tt}-\Delta u=0,&x\in\mb{R}^n,\ t>0,\\
(u,u_t)(0,x)=(u_0,u_1)(x),&x\in\mb{R}^n,
\end{cases}
\end{align}
with different regularity on the initial data $u_1$. Let us introduce the integral (0-th moment condition) of vector $\mathbf{w}_0=\mathbf{w}_0(x)$ by 
\begin{align*}
P_{\mathbf{w}_0}:=\left(\,\int_{\mb{R}^n}\mathbf{w}_0^{(1)}(x)\,\mathrm{d}x,\dots,\int_{\mb{R}^n}\mathbf{w}_0^{(n)}(x)\,\mathrm{d}x\right)\ \ \mbox{with}\ \ \mathbf{w}_0:=(\mathbf{w}_0^{(1)},\dots,\mathbf{w}_0^{(n)})\in\mb{R}^n.
\end{align*} 
Concerning $\nabla^su_1\in L^2\cap L^{1,1}$ with $s\geqslant0$, large time behavior, especially, the local and global (in time) stabilities of crucial quantity 
\begin{align*}
M(t):=\|u(t,\cdot)\|_{L^2}
\end{align*}
 being our first contribution is classified by the critical dimensions $n_0:=2s$ and $n_1:=2+2s$, respectively, as follows:
\begin{itemize}
	\item it blows up in finite time $M(t)=+\infty$ if $n\leqslant n_0$;
	\item it blows up in infinite time with polynomial rate $M(t)\approx t^{1+s-\frac{n}{2}}|P_{\nabla^s u_1}|$ if $n_0<n<n_1$;
	\item it blows up in infinite time with logarithmic rate $M(t)\approx\sqrt{\ln t}
	\,|P_{\nabla^s u_1}|$ if $n=n_1$;
	\item it is globally in time stable $M(t)\approx |P_{\nabla^s u_1}|$ for large time if $n>n_1$.
\end{itemize}
All these estimates are optimal up to unexpressed multiplicative constants.
If we drop the $L^{1,1}$ integrability of initial data, i.e. $u_1\in L^2$ only, and take a special form of it, this quantity grows polynomially $t^{1-\epsilon}\lesssim M(t)\lesssim t$ with a sufficiently small constant $\epsilon>0$ for all dimensions $n\geqslant 1$.

Let us state our second contribution. To stabilize the quantity $M(t)$ globally in time for any $n\geqslant 1$ (particularly, the lower dimensions $n\leqslant n_1$), we propose an additional condition $\nabla^su_1\in L^{1,\kappa}$ with $\lfloor\kappa\rfloor>2+s-\frac{n}{2}$ and the vanishing $|\alpha|$-th moment condition on $\nabla^s u_1$ up to $|\alpha|<\lfloor\kappa\rfloor$. Precisely, it holds that $M(t)\approx|P_{x^{\lfloor\kappa\rfloor-1}\nabla^su_1}|$ for large time. Nevertheless, it will blow up in infinite time if $\lfloor\kappa\rfloor=2+s-\frac{n}{2}$ for $1+s-\frac{n}{2}\in\mb{N}_0$. These results imply some influence from the regularity of initial data to sharp large time behavior of solutions.

The study of free wave equation is well-established in the last centuries from the viewpoint of scattering theory and/or local energy decay, for example, \cite{Morawetz=1961,Ralston=1969,Morawetz-Ralston-Strauss=1978,Melrose=1979,Lax-Phillips=1989,Ikehata=2004-wave} and a large number of references therein. In particular, it can be found from \cite{Morawetz=1961} that the $L^2$-boundedness of corresponding solution itself plays an important role in the estimates of local energy decay. We are just going to review some related works on the linear Cauchy problem \eqref{Eq-Waves}. It is well-known that the pioneering works \cite{Strichartz=1970,Peral=1980} derived some $L^p-L^q$ upper bound estimates of solution $u(t,\cdot)$ with a suitable range for $p,q$ (see later Remark \ref{Rem-Lp-Lq} or the summary in \cite[Chapter 16]{Ebert-Reissig=2018}). Concerning the radially symmetric initial data, these $L^p-L^q$ estimates for the radial solution were improved by \cite{Kubo-Kubota=1995,Kubo-Kubota=1998,Ebert-Kapp-Picon=2016}. The recent paper \cite{Ikehata=2023} re-considered the optimal $L^2$ estimates for the solution $u(t,\cdot)$ in lower dimensions $n\in\{1,2\}$. By assuming $u_1\in L^{1,1}$ with $|P_{u_1}|\neq0$, the author of \cite{Ikehata=2023} obtained optimal growth estimates for $n=1$ (with the polynomial rate) and $n=2$ (with the logarithmic rate) as $t\gg1$, which implies the infinite time blow-up for $n\leqslant 2$. However, the lower bound estimates for higher dimensions $n\geqslant 3$ are still unknown (surely, we know the boundedness from its upper side). The first question naturally arises:
\begin{center}
\textbf{Question I:} \emph{What are the large time optimal estimates for the higher dimensions $n\geqslant 3$?}
\end{center}  From the stability perspective, when $n\leqslant 2$ in \cite{Ikehata=2023} the quantity $M(t)=\|u(t,\cdot)\|_{L^2}$ with $|P_{u_1}|\neq0$ is not stable (with some growth properties) for large time, which arises another question: 
\begin{center}
\textbf{Question II:} \emph{What is the critical condition on regularity of initial data to stabilize $M(t)$?}
\end{center} Afterwards, these results were generalized to the free plate equation \cite{Ikehata=2024}, the critical Moore-Gibson-Thompson equation \cite{Chen=2024}, the fractional wave equation \cite{Ikeda-Zhao=2024}, the thermoelastic system of type II \cite{Chen-Ikehata=2024}. All these mentioned works \cite{Ikehata=2023,Ikehata=2024,Chen=2024,Ikeda-Zhao=2024,Chen-Ikehata=2024} are based on the $L^{1,1}$ data and studied lower dimensions, therefore, it seems interesting to investigate how does the regularity of initial data influence on large time behavior of solutions to these undamped evolution equations in $\mb{R}^n$ for any $n\geqslant 1$, e.g. the higher regular data $\nabla^su_1\in L^{1,\kappa}$ with $s\geqslant 0$, or the $L^2$ data without asking any $L^1$ integrability. Such a viewpoint from regularity gives a slight different aspect to the research that has been done on free waves so far.

To answer these above questions, we later obtain some basic lemmas in Section \ref{Section-Lemma} by using the Fourier splitting method. For the sharp upper bound estimates, some time-dependent splitting functions (see Figure \ref{imggg}) are introduced for a general kernel. Then, for the sharp lower bound estimates, we propose a new parameter-dependent exponential function for higher dimensions and shrink the domain of integration suitably in each dimension. We not only solve these problems for the free wave equation \eqref{Eq-Waves} in Section \ref{Section_Main_Results}, but also derive large time behavior for some related physical models as applications of our results, including the wave equation with scale-invariant terms, the free plate/beam equation, the undamped $\sigma$-evolution equation with fractional Laplacian, the critical Moore-Gibson-Thompson equation from linear acoustics in inviscid fluids, and the linearized compressible Euler system from fluid mechanics. These results include (for lower dimensions) as well as improve (for higher dimensions) \cite{Ikehata=2023,Ikehata=2024,Chen=2024}. The detailed introduction of these models will be shown in Section \ref{Section-Applications}.

We expect that our new techniques in deriving the neither growth nor decay (lower bound) estimates in Proposition \ref{Prop-Lower} can be widely applied in other evolution equations, for example, the fractional wave equation \cite[Theorem 1.1 for $n\geqslant 2$]{Ikeda-Zhao=2024} and the thermoelastic system of type II \cite[Question in Remark 2.9]{Chen-Ikehata=2024}. These studies need some modifications of our tools, which are beyond the scope in this paper.

\paragraph{\large Notation} The generic constants $c,C$ may vary from line to line but are independent of the time variable. We write $f\lesssim g$ if there exists a positive constant $C$ such that $f\leqslant Cg$, analogously for $f\gtrsim g$. The sharp relation $f\approx g$ holds if and only if $g\lesssim f\lesssim g$. We denote by $\lceil a\rceil:=\min\{A\in\mb{Z}:a\leqslant A\}$ the ceiling function, and $\lfloor a\rfloor:=\max\{A\in\mb{Z}:a\geqslant A\}$ the floor function. Let us define the weighted $L^1$ space by
\begin{align*}
	L^{1,\kappa}:=\big\{f\in L^1:\ \|f\|_{L^{1,\kappa}}:=\|(1+|x|)^{\kappa}f\|_{L^1}<+\infty \big\}\ \ \mbox{with}\ \ \kappa\geqslant0.
\end{align*} We now introduce the time-dependent function (that will be used in some large time optimal estimates of solutions) as follows:
\begin{align*}
\ml{D}_{n,\sigma,s}(t):=\begin{cases}
t^{1-\frac{n-2s}{2\sigma}}&\mbox{if}\ \ n\in(2s,2\sigma+2s),\\
\sqrt{\ln t}&\mbox{if}\ \ n=2\sigma+2s,\\
1&\mbox{if}\ \ n\in(2\sigma+2s,+\infty),
\end{cases}
\end{align*}
with $n\in\mb{N}_+$, $\sigma\geqslant 1$ and $s\geqslant 0$. Finally, the surface area of $n$-dimensional unit ball is simply denoted by $\omega_n:=\int_{|\omega|=1}\,\mathrm{d}\sigma_{\omega}$.

\section{Main results}\label{Section_Main_Results}\setcounter{equation}{0}
\hspace{5mm}Let us first state the large time  optimal  $L^2$ estimates of solution with the $\dot{H}^s_1$ (i.e. $\nabla^s L^1$) data, where one may notice some influence of regularity $s$ on the optimal growth rates for lower dimensions. The first critical dimension $n_0=2s$ is introduced for the local (in time) $L^2$-stability. It also verifies the second critical dimension $n_1=2+2s$ for large time stability, in other words, when $n\in(2+2s,+\infty)$ the solution in the $L^2$ norm is marginally stable without any decay property, nevertheless, when $n\in(2s,2+2s]$ it is instable with some optimal growth properties for large time.
\begin{theorem}\label{Thm-L1-data}
Let $0\leqslant s<\frac{n}{2}$. Let us assume $u_0\in L^2$ and $\nabla^s u_1\in L^2\cap L^1$ for the free wave equation \eqref{Eq-Waves}. Then, the solution satisfies the upper bound estimates
\begin{align*}
\|u(t,\cdot)\|_{L^2}\lesssim\|(u_0,\nabla^su_1)\|_{L^2\times (L^2\cap L^1)}\times\begin{cases}
t^{1+s-\frac{n}{2}}&\mbox{if}\ \ n\in(2s,2+2s),\\
\sqrt{\ln t}&\mbox{if}\ \ n=2+2s,\\
1&\mbox{if}\ \ n\in(2+2s,+\infty),
\end{cases}
\end{align*}
for large time $t\gg1$. Furthermore, by assuming $u_0\in L^2\cap L^1$ and $\nabla^s u_1\in L^{1,1}$ with $|P_{\nabla^s u_1}|\neq0$, then the solution satisfies the lower bound estimates
\begin{align*}
\|u(t,\cdot)\|_{L^2}\gtrsim|P_{\nabla^su_1}|\times\begin{cases}
	t^{1+s-\frac{n}{2}}&\mbox{if}\ \ n\in(2s,2+2s),\\
	\sqrt{\ln t}&\mbox{if}\ \ n=2+2s,\\
	1&\mbox{if}\ \ n\in(2+2s,+\infty),
\end{cases}
\end{align*}
for large time $t\gg1$. However, concerning the higher regularity $s\geqslant \frac{n}{2}$, the solution blows up in finite time in the sense that
\begin{align*}
\|u(t,\cdot)\|_{L^2}=+\infty\ \ \mbox{for}\ \ t\in(0,t_0]
\end{align*}
with any $t_0>0$, provided that $u_0\in L^2$ and $\nabla^s u_1\in L^{1,1}$ with $|P_{\nabla^s u_1}|\neq0$.
\end{theorem}

Let us summarize all derived estimates and recall $n_0=2s$ as well as $n_1=2+2s$. For large time $t\gg1$, Theorem \ref{Thm-L1-data} may be encapsulated by the optimal estimates as follows:
\begin{align}\label{Optimal-Est-Waves}
\|u(t,\cdot)\|_{L^2}\approx\begin{cases}
+\infty&\mbox{if}\ \ n\in[1,n_0],\\
t^{1+s-\frac{n}{2}}&\mbox{if}\ \ n\in(n_0,n_1),\\
\sqrt{\ln t}&\mbox{if}\ \ n=n_1,\\
1&\mbox{if}\ \ n\in(n_1,+\infty),
\end{cases}
\end{align}
with the crucial initial data $\nabla^s u_1\in L^{1,1}$ for $s\geqslant0$, provided that $|P_{\nabla^s u_1}|>0$. Concerning the classical case $s=0$, i.e. $u_1\in L^{1,1}$ with $|P_{u_1}|>0$, the large time optimal estimates \eqref{Optimal-Est-Waves} immediately turn into 
\begin{align*}
	\|u(t,\cdot)\|_{L^2}\approx\begin{cases}
		\sqrt{t}&\mbox{if}\ \ n=1,\\
		\sqrt{\ln t}&\mbox{if}\ \ n=2,\\
		1&\mbox{if}\ \ n\geqslant 3,
	\end{cases}
\end{align*}
which exactly coincide with the recent results \cite[Theorems 1.1 and 1.2]{Ikehata=2023} when $n\leqslant 2$, but the neither decay nor growth estimates when $n\geqslant 3$ are new, particularly, the lower bounds.
\begin{remark}\label{Rem-Lp-Lq}
Taking $u_0=0$ and $u_1\in L^p\cap L^{1,1}$ with $|P_{u_1}|>0$ in the free wave equation \eqref{Eq-Waves}, the pioneering works \cite{Strichartz=1970,Peral=1980} showed the following $L^p-L^q$ estimates:
\begin{align}\label{Lp-Lq-upper}
\|u(t,\cdot)\|_{L^q}\lesssim t^{1-\frac{n}{p}+\frac{n}{q}}\|u_1\|_{L^p},
\end{align}
where $(\frac{1}{p},\frac{1}{q})$ belongs to the admissible closed triangle $\triangle_{Q_1Q_2Q_3}$ with the vertices
\begin{align*}
Q_1:=\left(\frac{1}{2}+\frac{1}{n+1},\frac{1}{2}-\frac{1}{n+1}\right),\  Q_2:=\left(\frac{1}{2}-\frac{1}{n-1},\frac{1}{2}-\frac{1}{n-1}\right),\  Q_3:=\left(\frac{1}{2}+\frac{1}{n-1},\frac{1}{2}+\frac{1}{n-1}\right)
\end{align*}
when $n\geqslant 3$ and $Q_2:=(0,0)$, $Q_3:=(1,1)$ when $n\leqslant 2$. Focusing on the higher dimensions $n\geqslant 3$ by H\"older's inequality, combining with the lower bound estimates in Theorem \ref{Thm-L1-data}, one easily derives
\begin{align*}
|P_{u_1}|\lesssim\|u(t,\cdot)\|_{L^2}&\lesssim\|u(t,\cdot)\|_{L^m}\|u(t,\cdot)\|_{L^q}\\
&\lesssim t^{1-\frac{n}{p}+\frac{n}{m}}\|u_1\|_{L^p}\|u(t,\cdot)\|_{L^q},
\end{align*}
with $\frac{1}{2}=\frac{1}{m}+\frac{1}{q}$ (thus, $q\geqslant 2$) and $(\frac{1}{p},\frac{1}{m})\in\triangle_{Q_1Q_2Q_3}$. Namely,
\begin{align*}
\|u(t,\cdot)\|_{L^q}\gtrsim t^{-1+\frac{n}{p}-\frac{n}{2}+\frac{n}{q}}\frac{|P_{u_1}|}{\|u_1\|_{L^p}},
\end{align*}
if we assume $\|u_1\|_{L^p}\neq0$ additionally, carrying $(\frac{1}{p},\frac{1}{2}-\frac{1}{q})\in \triangle_{Q_1Q_2Q_3}$, which partly induces  large time growth estimates in the $L^q$ norm under the restriction
\begin{align*}
\left\{\left(\frac{1}{p},\frac{1}{q}\right)\in[0,1]^2:\  \left(\frac{1}{p},\frac{1}{2}-\frac{1}{q}\right)\in \triangle_{Q_1Q_2Q_3} \ \ \mbox{such that}\ \ \frac{1}{p}+\frac{1}{q}>\frac{1}{2}+\frac{1}{n} \right\}.
\end{align*}
 Associated with $(\frac{1}{p},\frac{1}{q}),(\frac{1}{p},\frac{1}{2}-\frac{1}{q})\in\triangle_{Q_1Q_2Q_3}$, we are able to claim
\begin{align*}
\|u(t,\cdot)\|_{L^q}\gtrsim t^{1-\frac{n}{p}+\frac{n}{q}}\frac{|P_{u_1}|}{\|u_1\|_{L^p}}\ \ \mbox{when}\ \ p=\frac{4n}{n+4}>1\ (n\geqslant 3).
\end{align*}
It partly verifies  the optimal rates of \eqref{Lp-Lq-upper} in the sense of $L^{\frac{4n}{n+4}}-L^q$ estimates as an application of Theorem \ref{Thm-L1-data}. We next state two typical examples of our claim, where in general we assume $u_0=0$, $u_1\in L^{\frac{4n}{n+4}}\cap L^{1,1}$, $|P_{u_1}|>0$ and $\|u_1\|_{L^{\frac{4n}{n+4}}}>0$ for the free wave equation \eqref{Eq-Waves} when $n\geqslant 3$.
\begin{itemize}
	\item When $n=3$, concerning $q\in[\frac{36}{11},\frac{36}{7}]$ the following optimal estimate:
	\begin{align*}
		t^{-\frac{3}{4}+\frac{3}{q}}\frac{|P_{u_1}|}{\|u_1\|_{L^{\frac{12}{7}}(\mb{R}^3)}}\lesssim \|u(t,\cdot)\|_{L^q(\mb{R}^3)}\lesssim t^{-\frac{3}{4}+\frac{3}{q}}\|u_1\|_{L^{\frac{12}{7}}(\mb{R}^3)}
	\end{align*}
	holds for $t\gg1$, which grows optimally when $q\in[\frac{36}{11},4)$; decays optimally when $q\in(4,\frac{36}{7}]$; is neither decay nor growth when $q=4$.
	\item When $n=4$, concerning $q=4$ the following optimal bounded estimate:
	\begin{align*}
		\frac{|P_{u_1}|}{\|u_1\|_{L^2(\mb{R}^4)}}\lesssim\|u(t,\cdot)\|_{L^4(\mb{R}^4)}\lesssim\|u_1\|_{L^2(\mb{R}^4)}
	\end{align*}
holds for $t\gg1$.
\end{itemize}
\end{remark}

By dropping the $L^1$ (or, further $L^{1,1}$) assumption for the initial data $u_1\in L^2$, we cannot expect non-growth estimates in general. This is caused by the small frequencies. We next provide an example with a special $L^2$ data, which implies the importance of $L^1$ integrable data to stabilize the solution for higher dimensions as $t\gg1$.
\begin{theorem}\label{Thm-L2-data}
 Let us assume $u_0\in L^2$ and $u_1\in L^2$ such that $|\widehat{u}_1(\xi)|\gtrsim|\xi|^{-\frac{n}{2}+\frac{\epsilon}{2}}$ as $|\xi|\in(0,1]$ with a sufficiently small constant $\epsilon>0$ for the free wave equation \eqref{Eq-Waves}. Then, the solution satisfies the almost optimal estimate
\begin{align*}
t^{1-\frac{\epsilon}{2}}\lesssim\|u(t,\cdot)\|_{L^2}\lesssim t\,\|(u_0,u_1)\|_{L^2\times L^2}
\end{align*}
for large time $t\gg1$.
\end{theorem}
\begin{remark}
	The assumption $|\widehat{u}_1(\xi)|^2\gtrsim|\xi|^{-n+\epsilon}$ for $|\xi|\in(0,1]$ does not contradict with $\widehat{u}_1(\xi)\in L^2$ for $\in\mb{R}^n$ due to $0<\epsilon\ll 1$, whose detail is referred to \eqref{L2data}. We conjecture that the large time optimal estimate $\|u(t,\cdot)\|_{L^2}\approx t$ will hold for more general $L^2$ data.
\end{remark}

In summary, we so far discover some influence of different regular data $\nabla^s u_1$ on the classical wave equation \eqref{Eq-Waves} in the whole space $\mb{R}^n$, particularly, large time behavior for the quantity $M(t)=\|u(t,\cdot)\|_{L^2(\mb{R}^n)}$. It seems interesting but still open to investigate the optimal regularity (or the minimum regularity) of initial data $\nabla^su_1$ to stabilize the crucial quantity $\|u(t,\cdot)\|_{L^2(\Omega)}$ for the classical wave equation \eqref{Eq-Waves} in $\Omega\subseteq\mb{R}^n$ for any $n\geqslant 1$, for example, the exterior domain case.

To end this section, we propose a possible answer in the $L^{1,\kappa}$ data framework for the last question (or Question II) in $\mb{R}^n$ for any $n\geqslant 1$. As we seen in Theorem \ref{Thm-L1-data}, for $n\in(2+2s,+\infty)$ the quantity $M(t)=\|u(t,\cdot)\|_{L^2}$ is global (in time) stable when we crucially assume $\nabla^su_1\in L^2\cap L^{1,1}$ and is optimally estimated by $M(t)\approx|P_{\nabla^s u_1}|$ provided that $|P_{\nabla^s u_1}|\neq0$. For this reason, we next study the remaining case $n\in[1,2+2s]$, which is instable from the last setting.
\begin{theorem}\label{Thm-stable}
Let $n\in[1,2+2s]$ with $s\geqslant0$. Let us assume $u_0\in L^2$ and $\nabla^s u_1\in L^2\cap L^{1,\kappa}$ carrying
\begin{align}\label{Stab-Con}
\lfloor\kappa\rfloor>2+s-\frac{n}{2}
\end{align} for the free wave equation \eqref{Eq-Waves}. Additionally, by assuming
\begin{align}\label{Data-stable}
	\int_{\mb{R}^n}x^{\alpha}\,\nabla^su_1(x)\,\mathrm{d}x\begin{cases}
		=0&\mbox{if}\ \ |\alpha|<\lfloor\kappa\rfloor-1,\\
		\neq0&\mbox{if}\ \ |\alpha|=\lfloor\kappa\rfloor-1,
	\end{cases}
\end{align}
then the solution satisfies the optimal estimate
\begin{align}\label{Est-04}
|P_{x^{\lfloor\kappa\rfloor-1}\nabla^su_1}|\lesssim\|u(t,\cdot)\|_{L^2}\lesssim \|u_0\|_{L^2}+\|\nabla^su_1\|_{L^2\cap L^{1,\kappa}}
\end{align}
for large time $t\gg1$.
\end{theorem}
It is easy to observe that the blow-up phenomena in finite time (when $n\in[1,n_0]$) and in infinite time (when $n\in(n_0,n_1]$) from the previous result \eqref{Optimal-Est-Waves} have been eliminated by the neither decay nor growth estimate \eqref{Est-04}, i.e. the $L^2$-stability, in Theorem \ref{Thm-stable} thanks to the stronger assumptions $\nabla^su_1\in L^{1,\kappa}$ and the vanishing moment condition \eqref{Data-stable}.

\begin{remark}
Concerning the classical case $s=0$, in the framework of $\kappa\in\mb{N}_0$ (and, in turn, $\min\kappa=2$), we may take the assumptions $u_0\in L^2$ and $u_1\in L^2\cap L^{1,2}$ with $|P_{u_1}|=0$ but $|P_{xu_1}|\neq0$ when $n\in[1,2]$. Then, we arrive at
\begin{align*}
	|P_{xu_1}|\lesssim \|u(t,\cdot)\|_{L^2}\lesssim \|u_0\|_{L^2}+\|u_1\|_{L^2\cap L^{1,2}},
\end{align*}
which overcomes the instability (optimal growths) results in \cite[Theorems 1.1 and 1.2]{Ikehata=2023} when $n\in[1,2]$. Furthermore, we may claim that in the $L^2\cap L^{1,\kappa}$ framework with $\kappa\in\mb{N}_0$, as $t\gg1$ the optimal regularity of initial data $u_1$ for stabilizing the crucial quantity $\|u(t,\cdot)\|_{L^2}$ to the free wave equation \eqref{Eq-Waves} is $\kappa=1$. To be specific, it is instable when $\kappa\in\{0,1\}$ (see \cite{Ikehata=2023} or Theorem \ref{Thm-L1-data} in details) but stable when $\kappa\in\{2,3,\dots\}$ for large time.
\end{remark} 

\begin{remark}
For the general case $s\geqslant0$, we in Remark \ref{Rem-critical} state that: for $\ml{B}:=\{(n,s):\ 1+s-\frac{n}{2}\in\mb{N}_0\}$, if $u_0\in L^2$ and $\nabla^s u_1\in L^2\cap L^{1,\kappa}$ carrying the limit value $\lfloor\kappa\rfloor=2+s-\frac{n}{2}$, then the solution will blow up in infinite time such that
\begin{align*}
	\|u(t,\cdot)\|_{L^2}\gtrsim \sqrt{\ln t}\,|P_{{x^{\lfloor\kappa\rfloor-1}\nabla^su_1}}|
\end{align*}
for large time $t\gg1$, provided that the same condition \eqref{Data-stable} holds. In other words, without requiring the condition \eqref{Stab-Con}, the quantity is not large time stable in general. In this sense, the last estimate and Theorem \ref{Thm-stable} show that in the framework of $L^2\cap L^{1,\kappa}$ data, the critical regularity (stabilization threshold) of initial data $\nabla^su_1$ to stabilize $\|u(t,\cdot)\|_{L^2}$ for the free wave equation \eqref{Eq-Waves} is 
\begin{align*}
	\lfloor\kappa\rfloor=2+s-\frac{n}{2}\ \ \mbox{when}\ \ n\in[1,2+2s]
\end{align*}
at least for $(n,s)\in\ml{B}$.
\end{remark}

\section{Estimates of solution in the $L^2$ norm}\setcounter{equation}{0}\label{Section-Lemma}
\hspace{5mm}To begin with this section, we put forward the $L^2$ norm for the general kernel via
\begin{align*}
\ml{I}_{n,\sigma,s}(t):=\left\|\widehat{\ml{K}}(t,\xi;\sigma)\,\widehat{v}_0(\xi)+\frac{\sin(|\xi|^{\sigma}t)}{|\xi|^{\sigma+s}}\,\widehat{v}_1(\xi)\right\|_{L^2}^2
\end{align*}
with $\sigma\geqslant 1$ and $s\geqslant 0$, where $|\widehat{\ml{K}}(t,\xi;\sigma)|\lesssim 1$ uniformly in $t,\xi$, and $\widehat{v}_j(\xi)$ denote the functions related to some initial data (that will be chosen later). Note that this kernel includes a large class of fundamental solutions to the undamped evolution equations associated with the oscillation effect (see Subsection \ref{Sub-Sec-Proof} as well as Section \ref{Section-Applications}, for example, $\widehat{\ml{K}}(t,\xi;1)=\cos(|\xi|t)$, $\widehat{v}_0(\xi)=\widehat{u}_0(\xi)$, $\widehat{v}_1(\xi)=|\xi|^s\widehat{u}_1(\xi)$ and $\sigma=1$ in the free wave equation). For this reason, we are going to investigate large time behavior for $\ml{I}_{n,\sigma,s}(t)$ deeply in the next part.

\subsection{Some preliminaries on the crucial kernel}
\hspace{5mm}We focus on the large time upper bound estimate firstly. The proof is motivated by \cite{Ikehata=2023,Ikehata=2024} and based on the Fourier splitting method with suitable time-dependent splitting functions. One may notice that $\ml{D}_{n,\sigma,s}(t)\gtrsim 1$ for large time $t\gg1$, namely, the time-dependent coefficient of $v_1$ in the upper bound estimates will play the determinant role in comparison with the one of $v_0$.
\begin{prop}\label{Prop-Upper}
Let $\sigma\geqslant 1$, $0\leqslant s<\frac{n}{2}$, and $|\widehat{\ml{K}}(t,\xi;\sigma)|\lesssim 1$ uniformly in $t,\xi$. Let us assume $v_0\in L^2$ and $v_1\in L^2\cap L^1$. Then, the time-dependent function $\ml{I}_{n,\sigma,s}(t)$ satisfies the following sharp upper bound estimates:
\begin{align*}
\ml{I}_{n,\sigma,s}(t)\lesssim\|v_0\|_{L^2}^2+[\ml{D}_{n,\sigma,s}(t)]^2\,\|v_1\|_{L^2\cap L^1}^2
\end{align*}
for large time $t\gg1$.
\end{prop}
\begin{proof}
For the higher dimensions $n\in(2\sigma+2s,+\infty)$ it is clear that
\begin{align*}
\ml{I}_{n,\sigma,s}(t)&=\int_{\mb{R}^n}\left|\widehat{\ml{K}}(t,\xi;\sigma)\,\widehat{v}_0(\xi)+\frac{\sin(|\xi|^{\sigma}t)}{|\xi|^{\sigma+s}}\,\widehat{v}_1(\xi)\right|^2\,\mathrm{d}\xi\\
&\lesssim\int_{\mb{R}^n}|\widehat{v}_0(\xi)|^2\,\mathrm{d}\xi+\int_{|\xi|\leqslant 1}|\xi|^{-2\sigma-2s}\,\mathrm{d}\xi\,\|\widehat{v}_1\|_{L^{\infty}}^2+\int_{|\xi|\geqslant 1}|\widehat{v}_1(\xi)|^2\,\mathrm{d}\xi\\
&\lesssim\|v_0\|_{L^2}^2+\|v_1\|_{L^2\cap L^1}^2,
\end{align*}
where we used the boundedness of sine function, $\|\widehat{v}_1\|_{L^{\infty}}\lesssim \|v_1\|_{L^1}$, and $\int_0^1r^{n-1-2\sigma-2s}\,\mathrm{d}r<+\infty$ due to $n-1-2\sigma-2s>-1$.

Indeed, we need detailed discussion in the lower dimensions $n\in(2s, 2\sigma+2s]$ for the part of $\widehat{v}_1(\xi)$. We may introduce the next separation by the time-dependent threshold $t^{-\frac{1}{\sigma}}$:
\begin{align*}
&\int_{\mb{R}^n}\frac{|\sin(|\xi|^{\sigma}t)|^2}{|\xi|^{2\sigma+2s}}\,|\widehat{v}_1(\xi)|^2\,\mathrm{d}\xi\\
&\qquad=\left(\int_{|\xi|\leqslant t^{-\frac{1}{\sigma}}}+\int_{|\xi|\geqslant t^{-\frac{1}{\sigma}}} \right)\frac{|\sin(|\xi|^{\sigma}t)|^2}{|\xi|^{2\sigma+2s}}\,|\widehat{v}_1(\xi)|^2\,\mathrm{d}\xi=:I_n^{(1)}(t)+I_n^{(2)}(t).
\end{align*}
From $|\sin(|\xi|^{\sigma}t)|\leqslant|\xi|^{\sigma}t$ and the polar coordinates with $n\in(2s,2\sigma+2s]$, one finds
\begin{align*}
I_n^{(1)}(t)&\lesssim t^2\|\widehat{v}_1\|_{L^{\infty}}^2\int_{|\xi|\leqslant t^{-\frac{1}{\sigma}}}\left|\frac{\sin(|\xi|^{\sigma}t)}{|\xi|^{\sigma}t}\right|^2|\xi|^{-2s}\,\mathrm{d}\xi\\
&\lesssim t^2\|v_1\|_{L^1}^2\int_0^{t^{-\frac{1}{\sigma}}}r^{n-1-2s}\,\mathrm{d}r\\
&\lesssim t^{2-\frac{n-2s}{\sigma}}\|v_1\|_{L^1}^2.
\end{align*}

For another, a further decomposition with the refined threshold $t^{-\frac{1}{\alpha_0}}$ carrying $\alpha_0:=\frac{\sigma(2\sigma+2s)}{2\sigma+2s-n}>\sigma$ for the sub-critical dimensions $n\in(2s,2\sigma+2s)$ implies
\begin{align*}
I_n^{(2)}(t)&=\left(\int_{t^{-\frac{1}{\sigma}}\leqslant |\xi|\leqslant t^{-\frac{1}{\alpha_0}}}+\int_{|\xi|\geqslant t^{-\frac{1}{\alpha_0}}}\right)\frac{|\sin(|\xi|^{\sigma}t)|^2}{|\xi|^{2\sigma+2s}}\,|\widehat{v}_1(\xi)|^2\,\mathrm{d}\xi\\
&\lesssim \int_{t^{-\frac{1}{\sigma}}\leqslant|\xi|\leqslant t^{-\frac{1}{\alpha_0}}}|\xi|^{-2\sigma-2s}\,\mathrm{d}\xi\,\|v_1\|_{L^1}^2+t^{\frac{2\sigma+2s}{\alpha_0}}\int_{|\xi|\geqslant t^{-\frac{1}{\alpha_0}}}|\sin(|\xi|^{\sigma}t)|^2\,|\widehat{v}_1(\xi)|^2\,\mathrm{d}\xi\\
&\lesssim \int_{t^{-\frac{1}{\sigma}}}^{t^{-\frac{1}{\alpha_0}}}r^{n-1-2\sigma-2s}\,\mathrm{d}r\,\|v_1\|_{L^1}^2+t^{\frac{2\sigma+2s}{\alpha_0}}\|v_1\|_{L^2}^2\\
&\lesssim \left(t^{-\frac{n-2\sigma-2s}{\sigma}}-t^{-\frac{n-2\sigma-2s}{\alpha_0}}\right)\|v_1\|_{L^1}^2+t^{-\frac{n-2\sigma-2s}{\sigma}}\|v_1\|_{L^2}^2\\
&\lesssim t^{2-\frac{n-2s}{\sigma}}\|v_1\|_{L^2\cap L^1}^2.
\end{align*}
 Consequently, it follows
\begin{align*}
\ml{I}_{n,\sigma,s}(t)\lesssim \|v_0\|_{L^2}^2+t^{2-\frac{n-2s}{\sigma}}\|v_1\|_{L^2\cap L^1}^2
\end{align*}
for any $n\in(2s,2\sigma+2s)$.

\begin{figure}[http]
	\centering
	\begin{tikzpicture}
		\draw[->] (-0.2,0) -- (5.8,0) node[below] {$|\xi|$};
		\draw[->] (0,-0.2) -- (0,4.4) node[left] {$t$};
		\node[left] at (0,-0.2) {{$0$}};
		\draw[color=black] plot[smooth, tension=.7] coordinates {(0.5,4) (0.7,2.04) (0.9,1.23) (1.1,0.82) (1.3,0.59) (1.6,0.39) (2,0.25) (2.5,0.16) (3,0.11) (3.5,0.08) (4,0.0625) (5,0.04)};
		\draw[color=black] plot[smooth, tension=.7] coordinates {(0.76,4) (0.8,3.05) (0.9,1.69) (1,1)};
		\node[left, color=black] at (4.4,0.4) {{$\longleftarrow$ $|\xi|=t^{-\frac{1}{\sigma}}$}};
		\node[left, color=black] at (3.4,3.6) {{$\longleftarrow$ $|\xi|=t^{-\frac{1}{\alpha_0}}$}};
		\node[left, color=black] at (2.6,2) {{III}};
		\node[left, color=black] at (0.9,3.8) {{II}};
		\node[left, color=black] at (0.6,0.6) {{I}};
		\node[left] at (6,-1) {{Separating lines when $n\in(2s,2\sigma+2s)$}};
		\draw[->] (7.8,0) -- (13.8,0) node[below] {$|\xi|$};
		\draw[->] (8,-0.2) -- (8,4.4) node[left] {$t$};
		\node[left] at (8,-0.2) {{$0$}};
		\draw[color=black] plot[smooth, tension=.7] coordinates {(8.5,4) (8.7,2.04) (8.9,1.23) (9.1,0.82) (9.3,0.59) (9.6,0.39) (10,0.25) (10.5,0.16) (11,0.11) (11.5,0.08) (12,0.0625) (13,0.04)};
		\draw[color=black] plot[smooth, tension=.7] coordinates {(8.94,4) (9,2.72) (9.1,1.86) (9.2,1.49) (9.3,1.31) (9.4,1.2) (9.5,1.14) (9.8,1.05) (10,1.03) (10.5,1.01) (12,1) (13,1)};
		\node[left, color=black] at (12.2,3.6) {{$\longleftarrow$ $|\xi|=(\ln t)^{-\frac{1}{\alpha_1}}$}};
		\node[left, color=black] at (12.4,0.4) {{$\longleftarrow$ $|\xi|=t^{-\frac{1}{\sigma}}$}};
        \node[left, color=black] at (11.8,2) {{III'}};
		\node[left, color=black] at (10.8,0.7) {{II'}};
		\node[left, color=black] at (8.6,0.6) {{I}};
		\node[left] at (14,-1) {{Separating lines when $n=2\sigma+2s$}};
	\end{tikzpicture}
	\caption{Different separating lines in lower dimensions}
	\label{imggg}
\end{figure}
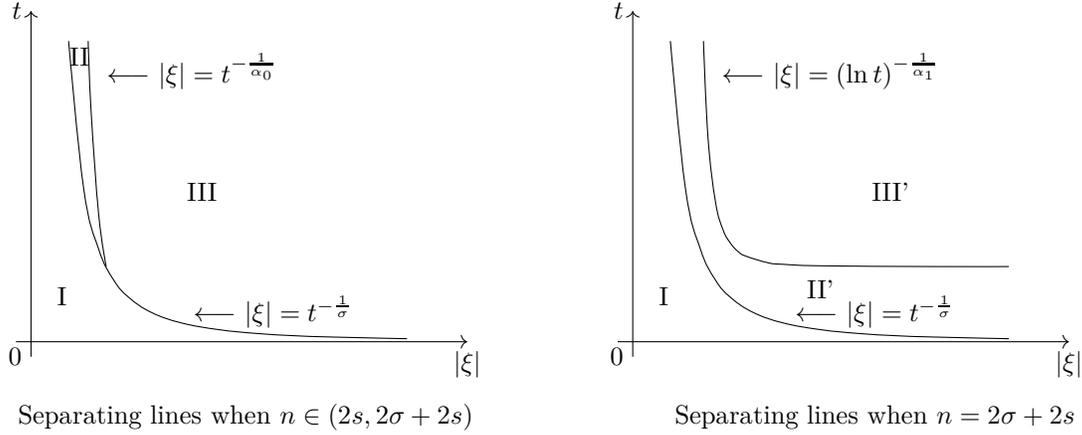
Finally, for the critical dimension $n=2\sigma+2s$ we have to improve the last estimate due to the power $(2-\frac{n-2s}{\sigma})|_{n=2\sigma+2s}=0$ for $I_n^{(2)}(t)$ by the Fourier splitting method with a new separation (see the separating lines in Figure \ref{imggg}), to be specific,
\begin{align*}
I_n^{(2)}(t)&=\left(\int_{t^{-\frac{1}{\sigma}}\leqslant |\xi|\leqslant (\ln t)^{-\frac{1}{\alpha_1}}}+\int_{|\xi|\geqslant (\ln t)^{-\frac{1}{\alpha_1}}}\right)\frac{|\sin(|\xi|^{\sigma}t)|^2}{|\xi|^{2\sigma+2s}}\,|\widehat{v}_1(\xi)|^2\,\mathrm{d}\xi\\
&\lesssim\int_{t^{-\frac{1}{\sigma}}}^{(\ln t)^{-\frac{1}{\alpha_1}}}r^{n-1-2\sigma-2s}\,\mathrm{d}r\,\|v_1\|_{L^1}^2+\int_{|\xi|\geqslant(\ln t)^{-\frac{1}{\alpha_1}}}|\xi|^{-2\sigma-2s}|\widehat{v}_1(\xi)|^2\,\mathrm{d}\xi\\
&\lesssim(\ln t-\ln\ln t)\|v_1\|_{L^1}^2+(\ln t)^{\frac{2\sigma+2s}{\alpha_1}}\|v_1\|_{L^2}^2\\
&\lesssim \ln t\,\|v_1\|_{L^2\cap L^1}^2
\end{align*}
for large time $t\gg1$ with the positive constant $\alpha_1:=2\sigma+2s\geqslant2\sigma>0$, which turns into 
\begin{align*}
\ml{I}_{n,\sigma,s}(t)\lesssim \|v_0\|_{L^2}^2+\ln t\,\|v_1\|_{L^2\cap L^1}^2
\end{align*}
for $n=2\sigma+2s$. The proof is completed.
\end{proof}

To ensure the large time optimality of derived estimates in the last proposition when $0\leqslant s<\frac{n}{2}$, we next study their lower bounds. We underline that the bounded lower bounds for the higher dimensions $n\in(2\sigma+2s,+\infty)$ do not been considered even in the simplest case $\sigma=1$ and $s=0$. Different from the exponential (in time) decay kernel of strongly damped waves \cite{Ikehata=2014,Ikehata-Ono=2017} or the growth (in time) functions in lower dimensions of the free waves \cite{Ikehata=2023}, we are going to introduce a parameter$(\gamma)$-dependent exponential function $\mathrm{e}^{-\gamma|\xi|^2}$ to control the remainder terms by suitable large $\gamma>0$. Then, associated with the suitable shrinking domains, it allows us to obtain the boundedness of $\ml{I}_{n,\sigma,s}(t)$ from the below side.
\begin{prop}\label{Prop-Lower}
Let $\sigma\geqslant 1$, $0\leqslant s<\frac{n}{2}$, and $|\widehat{\ml{K}}(t,\xi;\sigma)|\lesssim 1$ uniformly in $t,\xi$. Let us assume $v_0\in L^2\cap L^1$ and $v_1\in L^{1,1}$ with $|P_{v_1}|\neq0$. Then, the time-dependent function $\ml{I}_{n,\sigma,s}(t)$ satisfies the following lower bound estimates:
\begin{align*}
	\ml{I}_{n,\sigma,s}(t)\gtrsim [\ml{D}_{n,\sigma,s}(t)]^2\,|P_{v_1}|^2
\end{align*}
for large time $t\gg1$.
\end{prop}
\begin{remark}
We need $v_0\in L^2$ only as usual in the lower dimensions $n\in(2s,2\sigma+2s]$. For the higher dimensions we technically require $v_0\in L^1$.
\end{remark}

\begin{proof}
Let us consider the lower dimensional case $n\in(2s,2\sigma+2s)$. It is known from the proof of Proposition \ref{Prop-Upper}, the zone-I (see Figure \ref{imggg}) has the dominant influence for large time. By using the triangle inequality $|f-g|^2\geqslant\frac{1}{2}\,|f|^2-|g|^2$ and shrinking the domain of integration from $\mb{R}^n$ to $\{|\xi|\leqslant \delta_0 t^{-\frac{1}{\sigma}}\}$ with a suitably small constant $\delta_0>0$, it may follow that
\begin{align*}
\ml{I}_{n,\sigma,s}(t)&\geqslant\frac{1}{2}\int_{|\xi|\leqslant\delta_0t^{-\frac{1}{\sigma}}}\frac{|\sin(|\xi|^{\sigma}t)|^2}{|\xi|^{2\sigma+2s}}\,|\widehat{v}_1(\xi)|^2\,\mathrm{d}\xi-\int_{|\xi|\leqslant \delta_0t^{-\frac{1}{\sigma}}}|\widehat{v}_0(\xi)|^2\,\mathrm{d}\xi\\
&\gtrsim|P_{v_1}|^2\int_{|\xi|\leqslant\delta_0 t^{-\frac{1}{\sigma}}}\frac{|\sin(|\xi|^{\sigma}t)|^2}{|\xi|^{2\sigma+2s}}\,\mathrm{d}\xi-\|v_1\|_{L^{1,1}}^2\int_{|\xi|\leqslant\delta_0 t^{-\frac{1}{\sigma}}}\frac{|\sin(|\xi|^{\sigma}t)|^2}{|\xi|^{2\sigma+2s-2}}\,\mathrm{d}\xi-\|v_0\|_{L^2}^2,
\end{align*}
where thanks to $v_1\in L^{1,1}$ and $|P_{v_1}|\neq0$ we used (from \cite[Lemma 3.1]{Ikehata=2004})
\begin{align}\label{Ineq-01}
|\widehat{v}_1(\xi)|^2&=|P_{v_1}+\widehat{v}_1(\xi)-P_{v_1}|^2\geqslant\frac{1}{2}|P_{v_1}|^2-|\widehat{v}_1(\xi)-P_{v_1}|^2\notag\\
&\gtrsim|P_{v_1}|^2-|\xi|^2\|v_1\|_{L^{1,1}}^2.
\end{align}
For one thing, by using the fact that
\begin{align*}
\left|\frac{\sin(|\xi|^{\sigma}t)}{|\xi|^{\sigma}t}\right|\geqslant\frac{1}{\sqrt{2}}\ \ \mbox{when}\ \ |\xi|^{\sigma}t\leqslant \delta_0^{\sigma}\ll 1,
\end{align*}
 the lower bound can be controlled by
\begin{align*}
\int_{|\xi|\leqslant \delta_0 t^{-\frac{1}{\sigma}}}\frac{|\sin(|\xi|^{\sigma}t)|^2}{|\xi|^{2\sigma+2s}}\,\mathrm{d}\xi&\gtrsim t^2\int_{|\xi|\leqslant \delta_0 t^{-\frac{1}{\sigma}}}|\xi|^{-2s}\,\mathrm{d}\xi\\
&\gtrsim t^2\int_0^{\delta_0t^{-\frac{1}{\sigma}}}r^{n-1-2s}\,\mathrm{d}r\gtrsim t^{2-\frac{n-2s}{\sigma}}
\end{align*}
due to $n-2s>0$.
For another, the same philosophy indicates
\begin{align*}
\int_{|\xi|\leqslant \delta_0 t^{-\frac{1}{\sigma}}}\frac{|\sin(|\xi|^{\sigma}t)|^2}{|\xi|^{2\sigma+2s-2}}\,\mathrm{d}\xi\lesssim t^2\int_{|\xi|\leqslant \delta_0t^{-\frac{1}{\sigma}}}|\xi|^{2-2s}\,\mathrm{d}\xi\lesssim t^{2-\frac{n+2-2s}{\sigma}},
\end{align*}
which demonstrates immediately
\begin{align*}
\ml{I}_{n,\sigma,s}(t)\gtrsim t^{2-\frac{n-2s}{\sigma}}\left(|P_{v_1}|^2-t^{-\frac{2}{\sigma}}\|v_1\|_{L^{1,1}}^2\right)-\|v_0\|_{L^2}^2\gtrsim t^{2-\frac{n-2s}{\sigma}}|P_{v_1}|^2
\end{align*}
for large time $t\gg1$, provided that $|P_{v_1}|\neq0$, where we recall that $2-\frac{n-2s}{\sigma}>0$ is the optimal growth rate for $n\in(2s,2\sigma+2s)$.

Next, we turn to the critical dimensional case $n=2\sigma+2s$. Thanks to $1\geqslant \mathrm{e}^{-|\xi|^2}$ for any $\xi\in\mb{R}^n$, by the triangle inequality as the lower dimensions, one can get
\begin{align}
\ml{I}_{n,\sigma,s}(t)&\geqslant\frac{1}{2}\int_{\mb{R}^n}\mathrm{e}^{-|\xi|^2}\,\frac{|\sin(|\xi|^{\sigma}t)|^2}{|\xi|^{2\sigma+2s}}\,|\widehat{v}_1(\xi)|^2\,\mathrm{d}\xi-\|v_0\|_{L^2}^2\notag\\
&\gtrsim|P_{v_1}|^2\int_{\mb{R}^n}\mathrm{e}^{-|\xi|^2}\,\frac{|\sin(|\xi|^{\sigma}t)|^2}{|\xi|^{2\sigma+2s}}\,\mathrm{d}\xi-\|v_1\|_{L^{1,1}}^2\int_{\mb{R}^n}\mathrm{e}^{-|\xi|^2}\,\frac{|\sin(|\xi|^{\sigma}t)|^2}{|\xi|^{2\sigma+2s-2}}\,\mathrm{d}\xi-\|v_0\|_{L^2}^2.\label{Est-05}
\end{align}
Clearly, with the aid of $n=2\sigma+2s$, it holds
\begin{align*}
\int_{\mb{R}^n}\mathrm{e}^{-|\xi|^2}\,\frac{|\sin(|\xi|^{\sigma}t)|^2}{|\xi|^{2\sigma+2s-2}}\,\mathrm{d}\xi&\lesssim\int_0^{+\infty}\mathrm{e}^{-r^2}\,|\sin(r^{\sigma}t)|^2\,r^{n+1-2\sigma-2s}\,\mathrm{d}r\\
&\lesssim\int_0^{+\infty}\mathrm{e}^{-r^2}\,r\,\mathrm{d}r<+\infty.
\end{align*}
Let us introduce the $j$-dependent sequences (motivated by \cite{Ikehata-Ono=2017,Ikehata=2023,Ikehata=2024})
\begin{align*}
\nu_j:=\left[\left(\frac{1}{4}+j\right)\frac{\pi}{t}\right]^{\frac{1}{\sigma}}\ \ \mbox{and}\ \ \mu_j:=\left[\left(\frac{3}{4}+j\right)\frac{\pi}{t}\right]^{\frac{1}{\sigma}} \ \ \mbox{for any}\ \ j\in\mb{N}_0,
\end{align*}
where their error (the length of $[\nu_j,\mu_j]$), i.e.
\begin{align}\label{Seq-02}
\mu_j-\nu_j=\left(\frac{\pi}{t}\right)^{\frac{1}{\sigma}}\left[\left(\frac{3}{4}+j\right)^{\frac{1}{\sigma}}-\left(\frac{1}{4}+j\right)^{\frac{1}{\sigma}}\right]=\frac{1}{2\sigma}\left(\frac{\pi}{t}\right)^{\frac{1}{\sigma}}(\theta_1+j)^{\frac{1}{\sigma}-1}
\end{align}
with $\theta_1\in(\frac{1}{4},\frac{3}{4})$ by the Lagrange mean value theorem, is a non-increasing function with respect to $j$ because of $\sigma\geqslant 1$. Note that the length is fixed if $\sigma=1$ and decays to zero if $\sigma>1$.
Then, for any $|\xi|\in[\nu_j,\mu_j]$ with $j\in\mb{N}_0$, namely, $(\frac{1}{4}+j)\pi\leqslant |\xi|^{\sigma}t\leqslant(\frac{3}{4}+j)\pi$, it yields
\begin{align}\label{Treatment-sine}
1\geqslant|\sin(|\xi|^{\sigma}t)|\geqslant\frac{1}{\sqrt{2}}.
\end{align}
By shrinking the domain of integration from $\mb{R}^n$ to $\cup_{j=0}^{+\infty}\,[\nu_j,\mu_j]$ and using the last lower bound estimate, one may deduce
\begin{align}\label{Est-07}
\int_{\mb{R}^n}\mathrm{e}^{-|\xi|^2}\,\frac{|\sin(|\xi|^{\sigma}t)|^2}{|\xi|^{2\sigma+2s}}\,\mathrm{d}\xi
&\geqslant\sum\limits_{j=0}^{+\infty}\int_{\nu_j\leqslant |\xi|\leqslant\mu_j}\mathrm{e}^{-|\xi|^2}\left(\frac{1}{\sqrt{2}}\right)^2|\xi|^{-2\sigma-2s}\,\mathrm{d}\xi\notag\\
&\geqslant\frac{\omega_n}{2}\sum\limits_{j=0}^{+\infty}\int_{\nu_j}^{\mu_j}\mathrm{e}^{-r^2}\,r^{-1}\,\mathrm{d}r\notag\\
&\geqslant \frac{\omega_n}{4}\int_{\nu_0}^{+\infty}\mathrm{e}^{-r^2}\,r^{-1}\,\mathrm{d}r.
\end{align}
In the last line, we employed the monotone decreasing property of integrand $\mathrm{e}^{-r^2}\,r^{-1}$ and the fixed ($\sigma=1$) or decay ($\sigma>1$) length of $[\nu_j,\mu_j]$ arising from \eqref{Seq-02} so that
\begin{align*}
\int_{\nu_j}^{\mu_j}\mathrm{e}^{-r^2}\,r^{-1}\,\mathrm{d}r\geqslant\frac{1}{2}\int_{\nu_j}^{\nu_{j+1}}\mathrm{e}^{-r^2}\,r^{-1}\,\mathrm{d}r\ \ \mbox{for any}\ \ j\in\mb{N}_0.
\end{align*}
 By shrinking the domain again, one notices
\begin{align*}
\frac{\omega_n}{4}\int_{\nu_0}^{+\infty}\mathrm{e}^{-r^2}\,r^{-1}\,\mathrm{d}r&\geqslant\frac{\omega_n}{4}\int_{(\frac{\pi}{4})^{\frac{1}{\sigma}}t^{-\frac{1}{\sigma}}}^1\mathrm{e}^{-r^2}\,r^{-1}\,\mathrm{d}r\\
&\geqslant\frac{\omega_n}{4\sigma\mathrm{e}}\left(\ln\frac{4}{\pi}+\ln t\right)\gtrsim \ln t
\end{align*}
for large time $t\gg1$. In conclusion, the sharp large time lower bound estimate
\begin{align*}
\ml{I}_{n,\sigma,s}(t)\gtrsim\ln t\,|P_{v_1}|^2-\|(v_0,v_1)\|_{L^2\times L^{1,1}}^2\gtrsim \ln t\,|P_{v_1}|^2
\end{align*}
holds for $n=2\sigma+2s$, which is the optimal growth rate comparing with the one in Proposition \ref{Prop-Upper} thanks to $|P_{v_1}|\neq0$.

Eventually, we study the higher dimensional case $n\in(2\sigma+2s,+\infty)$, in which we cannot expect any $L^2$-decay estimate due to the lack of damping mechanism in the free waves (different from the strongly damped waves \cite{Ikehata=2014,Ikehata-Ono=2017}). We now introduce the exponential factor according to 
\begin{align*}
1\geqslant \mathrm{e}^{-\gamma|\xi|^2}\ \ \mbox{for any}\ \ \xi\in\mb{R}^n
\end{align*}
 with a positive (suitably large) parameter $\gamma>0$ to be chosen later. Precisely, the triangle inequality implies
\begin{align}\label{Est-01}
\ml{I}_{n,\sigma,s}(t)&\geqslant\int_{\mb{R}^n}\mathrm{e}^{-\gamma|\xi|^2}\left|\widehat{\ml{K}}(t,\xi;\sigma)\,\widehat{v}_0(\xi)+\frac{\sin(|\xi|^{\sigma}t)}{|\xi|^{\sigma+s}}\,\widehat{v}_1(\xi)\right|^2\,\mathrm{d}\xi\notag\\
&\geqslant\frac{1}{2}\int_{\mb{R}^n}\mathrm{e}^{-\gamma|\xi|^2}\,\frac{|\sin(|\xi|^{\sigma}t)|^2}{|\xi|^{2\sigma+2s}}\,|\widehat{v}_1(\xi)|^2\,\mathrm{d}\xi-\int_{\mb{R}^n}\mathrm{e}^{-\gamma|\xi|^2}\,|\widehat{\ml{K}}(t,\xi;\sigma)|^2\,|\widehat{v}_0(\xi)|^2\,\mathrm{d}\xi\notag\\
&=:I_n^{(3)}(t;\gamma)+I_n^{(4)}(t;\gamma).
\end{align}
By our assumption on $\widehat{\ml{K}}(t,\xi;\sigma)$ there exists a positive constant $C_0=C_0(n,\sigma)$ but independent of $\gamma$ and $t$ such that
\begin{align*}
I_n^{(4)}(t;\gamma)\leqslant C_0\int_{\mb{R}^n}\mathrm{e}^{-\gamma|\xi|^2}\,\mathrm{d}\xi\,\|v_0\|_{L^1}^2&\leqslant C_0\,\omega_n\|v_0\|_{L^1}^2\int_0^{+\infty}\mathrm{e}^{-\gamma r^2}\,r^{n-1}\,\mathrm{d}r\\
&\leqslant \left(C_0\,\omega_n\|v_0\|_{L^1}^2\int_0^{+\infty}\mathrm{e}^{-\eta^2}\,\eta^{n-1}\,\mathrm{d}\eta\right)\gamma^{-\frac{n}{2}}.
\end{align*}
Actually, from \eqref{Ineq-01} there exists a positive constant $M$ such that
\begin{align*}
I_n^{(3)}(t;\gamma)&\geqslant\frac{|P_{v_1}|^2}{4}\int_{\mb{R}^n}\mathrm{e}^{-\gamma|\xi|^2}\,\frac{|\sin(|\xi|^{\sigma}t)|^2}{|\xi|^{2\sigma+2s}}\,\mathrm{d}\xi-\frac{M^2}{2}\,\|v_1\|_{L^{1,1}}^2\int_{\mb{R}^n}\mathrm{e}^{-\gamma|\xi|^2}\,\frac{|\sin(|\xi|^{\sigma}t)|^2}{|\xi|^{2\sigma+2s-2}}\,\mathrm{d}\xi\\
&\geqslant\frac{|P_{v_1}|^2}{4}\sum\limits_{j=0}^{+\infty}\int_{\nu_j\leqslant |\xi|\leqslant\mu_j}\mathrm{e}^{-\gamma|\xi|^2}\left(\frac{1}{\sqrt{2}}\right)^2 |\xi|^{-2\sigma-2s}\,\mathrm{d}\xi-\frac{M^2}{2}\,\|v_1\|_{L^{1,1}}^2\int_{\mb{R}^n}\mathrm{e}^{-\gamma|\xi|^2}\,|\xi|^{2-2\sigma-2s}\,\mathrm{d}\xi,
\end{align*}
where we used two different treatments of sine function via \eqref{Treatment-sine}.
Note that
\begin{align*}
\int_0^{+\infty}\mathrm{e}^{-\gamma r^2}\,r^{n+1-2\sigma-2s}\,\mathrm{d}r=\gamma^{-\frac{n+2-2\sigma-2s}{2}}\int_0^{+\infty}\mathrm{e}^{-\eta^2}\,\eta^{n+1-2\sigma-2s}\,\mathrm{d}\eta.
\end{align*}
Moreover, thanks to
\begin{align*}
\frac{\mathrm{d}}{\mathrm{d}r}\left(\mathrm{e}^{-\gamma r^2}\,r^{n-1-2\sigma-2s}\right)=\mathrm{e}^{-\gamma r^2}\,r^{n-2-2\sigma-2s}\big((n-1-2\sigma-2s)-2\gamma r^2\big),
\end{align*}
that is to say, when
\begin{align*}
r>\sqrt{\frac{\max\{n-1-2\sigma-2s,0\}}{2\gamma}}=:\theta_0=\theta_0(n,\sigma,s,\gamma),
\end{align*}
 we claim that the last derivative is negative and therefore $\mathrm{e}^{-\gamma r^2}\,r^{n-1-2\sigma-2s}$ is a monotone decreasing function for large $r$. Let us take the index $j_1=j_1(t,n,\sigma,s,\gamma)\in\mb{N}_0$ via
\begin{align*}
j_1:=\left\lceil\theta_0^{\sigma}\frac{t}{\pi}-\frac{1}{4}\right\rceil\ \ \Rightarrow \ \ \nu_{j_1}=\left[\left(\frac{1}{4}+j_1\right)\frac{\pi}{t}\right]^{\frac{1}{\sigma}}\geqslant\theta_0.
\end{align*}
Therefore, it follows
\begin{align}\label{Est-02}
\sum\limits_{j=0}^{+\infty}\int_{\nu_j\leqslant |\xi|\leqslant\mu_j}\mathrm{e}^{-\gamma|\xi|^2}|\xi|^{-2\sigma-2s}\,\mathrm{d}\xi
&\geqslant \omega_n \sum\limits_{j=j_1}^{+\infty}\int_{\nu_j}^{\mu_j}\mathrm{e}^{-\gamma r^2}\,r^{n-1-2\sigma-2s}\,\mathrm{d}r\notag\\
&\geqslant\frac{\omega_n}{2}\int_{\nu_{j_1}}^{+\infty}\mathrm{e}^{-\gamma r^2}\,r^{n-1-2\sigma-2s}\,\mathrm{d}r=:I_n^{(5)}(t;\gamma),
\end{align}
where we used \eqref{Seq-02} again and the non-increasing property for $\mathrm{e}^{-\gamma r^2}\,r^{n-1-2\sigma-2s}$ as $r\geqslant \nu_{j_1}$.
\begin{itemize}
	\item When $n\in(2\sigma+2s,2\sigma+2s+1]$, it is easy to check $\theta_0=0$, and, in turn, $j_1=0$. Let us shrink the domain of integration from $[\nu_0,+\infty]$ to $[\gamma^{-{1/2}},2\gamma^{-{1/2}}]$ so that
	\begin{align*}
		I_n^{(5)}(t;\gamma)&=\frac{\omega_n}{2}\int_{\nu_0}^{+\infty}\mathrm{e}^{-\gamma r^2}\,r^{n-1-2\sigma-2s}\,\mathrm{d}r\geqslant\frac{\omega_n}{2}\int_{\gamma^{-\frac{1}{2}}}^{2\gamma^{-\frac{1}{2}}}\mathrm{e}^{-\gamma r^2}\,r^{n-1-2\sigma-2s}\,\mathrm{d}r\\
		&\geqslant\frac{\omega_n}{2\,\mathrm{e}^4}\int_{\gamma^{-\frac{1}{2}}}^{2\gamma^{-\frac{1}{2}}}r^{n-1-2\sigma-2s}\,\mathrm{d}r=\frac{\omega_n(2^{n-2\sigma-2s}-1)}{2(n-2\sigma-2s)\,\mathrm{e}^4}\,\gamma^{-\frac{n-2\sigma-2s}{2}}
	\end{align*}
for large time $t\gg1$ so that $\nu_0=(\frac{\pi}{4t})^{1/\sigma}\leqslant \gamma^{-1/2}$ since $\gamma$ is a fixed constant.
\item When $n\in(2\sigma+2s+1,+\infty)$, i.e. $\theta_0>0$, one notes that
\begin{align*}
	\nu_{j_1}^{\sigma}&\leqslant\frac{\pi}{4t}+\left(\theta_0^{\sigma}\frac{t}{\pi}+\frac{3}{4}\right)\frac{\pi}{t}= \frac{\pi}{t}+\theta_0^{\sigma}\leqslant 2^{\sigma}\theta_0^{\sigma}
\end{align*}
for large time $t\gg1$ such that $t\geqslant\frac{\pi}{(2^{\sigma}-1)\theta_0^{\sigma}}$. By shrinking the  domain of integration again from $[\nu_{j_1},+\infty]$ to $[2\theta_0,4\theta_0]$, one obtains immediately
\begin{align*}
	I_n^{(5)}(t;\gamma)&\geqslant\frac{\omega_n}{2}\int_{2\theta_0}^{4\theta_0}\mathrm{e}^{-\gamma r^2}\,r^{n-1-2\sigma-2s}\,\mathrm{d}r\\
	&\geqslant\frac{\omega_n(4^{n-2\sigma-2s}-2^{n-2\sigma-2s})}{2(n-2\sigma-2s)\,\mathrm{e}^{8(n-2\sigma-2s-1)}}\left(\frac{n-2\sigma-2s-1}{2}\right)^{\frac{n-2\sigma-2s}{2}}\gamma^{-\frac{n-2\sigma-2s}{2}}.
\end{align*}
\end{itemize} 
Thus, concerning $t\gg1$, for any $n\in(2\sigma+2s,+\infty)$ there exists a positive constant $C_1=C_1(n,\sigma,s)$ but independent of $\gamma$ and $t$ such that 
\begin{align*}
I_n^{(5)}(t;\gamma)&\geqslant C_1\,\omega_n\gamma^{-\frac{n-2\sigma-2s}{2}}.
\end{align*}
In conclusion, from \eqref{Est-01} we are able to claim
\begin{align*}
\ml{I}_{n,\sigma,s}(t)&\geqslant\frac{C_1\,\omega_n|P_{v_1}|^2}{8}\,\gamma^{-\frac{n-2\sigma-2s}{2}}-\frac{M^2\omega_n\|v_1\|_{L^{1,1}}^2}{2}\left(\int_0^{+\infty}\mathrm{e}^{-\eta^2}\,\eta^{n+1-2\sigma-2s}\,\mathrm{d}\eta\right)\gamma^{-\frac{n+2-2\sigma-2s}{2}}\\
&\quad\  -\left(C_0\,\omega_n\|v_0\|_{L^1}^2\int_0^{+\infty}\mathrm{e}^{-\eta^2}\,\eta^{n-1}\,\mathrm{d}\eta\right)\gamma^{-\frac{n}{2}}\\
&\geqslant\omega_n\gamma^{-\frac{n-2\sigma-2s}{2}}\left(\frac{C_1|P_{v_1}|^2}{8}-\frac{M^2\|v_1\|_{L^{1,1}}^2}{2\gamma}\int_0^{+\infty}\mathrm{e}^{-\eta^2}\,\eta^{n+1-2\sigma-2s}\,\mathrm{d}\eta\right.\\
&\left.\qquad\qquad\qquad\quad\ -\frac{C_0\|v_0\|_{L^1}^2}{\gamma^{\sigma+s}}\int_0^{+\infty}\mathrm{e}^{-\eta^2}\,\eta^{n-1}\,\mathrm{d}\eta\right)\\
&\geqslant \frac{C_1\,\omega_n}{16}\,\gamma^{-\frac{n-2\sigma-2s}{2}}|P_{v_1}|^2
\end{align*}
by taking suitably large (but fixed) constant $\gamma>0$ such that
\begin{align*}
\frac{C_1|P_{v_1}|^2}{16}-\frac{M^2\|v_1\|_{L^{1,1}}^2}{2\gamma}\int_0^{+\infty}\mathrm{e}^{-\eta^2}\,\eta^{n+1-2\sigma-2s}\,\mathrm{d}\eta-\frac{C_0\|v_0\|_{L^1}^2}{\gamma^{\sigma+s}}\int_0^{+\infty}\mathrm{e}^{-\eta^2}\,\eta^{n-1}\,\mathrm{d}\eta\geqslant0,
\end{align*}
due to $|P_{v_1}|\neq0$, $v_0\in L^1$, $v_1\in L^{1,1}$ and
\begin{align*}
\int_0^{+\infty}\mathrm{e}^{-\eta^2}\,\eta^{n+1-2\sigma-2s}\,\mathrm{d}\eta<+\infty,\ \ \int_0^{+\infty}\mathrm{e}^{-\eta^2}\,\eta^{n-1}\,\mathrm{d}\eta<+\infty.
\end{align*}
For example, we may choose $\gamma=\gamma(n,\sigma,s,\|v_0\|_{L^1},\|v_1\|_{L^{1,1}},|P_{v_1}|)$ by
\begin{align}\label{Est-03}
\gamma=\max\left\{\frac{16M^2\|v_1\|_{L^{1,1}}^2}{C_1|P_{v_1}|^2}\int_0^{+\infty}\mathrm{e}^{-\eta^2}\,\eta^{n+1-2\sigma-2s}\,\mathrm{d}\eta,\left[\frac{32C_0\|v_0\|_{L^1}^2}{C_1|P_{v_1}|^2}\int_0^{+\infty}\mathrm{e}^{-\eta^2}\,\eta^{n-1}\,\mathrm{d}\eta \right]^{\frac{1}{\sigma+s}} \right\}.
\end{align}
So, our proof of this proposition is complete.
\end{proof}

Let us turn to the remaining case $n\in[1,2s]$, which leads to the very strong singularity $|\xi|^{-(\sigma+s)}$ in lower dimensions even for fixed time. Therefore, even with the additional $L^{1,1}$ regularity on the initial data $v_1$ (but $|P_{v_1}|>0$), the quantity $\ml{I}_{n,\sigma,s}(t)$ cannot be locally (in time) defined with the stronger singularity $s\geqslant \frac{n}{2}$.
\begin{prop}\label{Prop-Blow-up}
Let $\sigma\geqslant 1$, $s\geqslant \frac{n}{2}$, and $|\widehat{\ml{K}}(t,\xi;\sigma)|\lesssim 1$ uniformly in $t,\xi$. Let us assume $v_0\in L^2$ and $v_1\in L^{1,1}$ with $|P_{v_1}|\neq0$. Then, the time-dependent function $\ml{I}_{n,\sigma,s}(t)$ blows up in finite time, precisely,
\begin{align*}
\ml{I}_{n,\sigma,s}(t)=+\infty\ \ \mbox{for}\ \ t\in(0,t_0]
\end{align*}
with any $t_0>0$.
\end{prop}
\begin{proof}
Again from the triangle inequality as before, one notices
\begin{align*}
\ml{I}_{n,\sigma,s}(t)&\gtrsim|P_{v_1}|^2\int_{|\xi|\in[\varepsilon_0,\varepsilon_1]}\frac{|\sin(|\xi|^{\sigma}t)|^2}{|\xi|^{2\sigma+2s}}\,\mathrm{d}\xi-\|v_1\|_{L^{1,1}}^2\int_{|\xi|\in[\varepsilon_0,\varepsilon_1]}\frac{|\sin(|\xi|^{\sigma}t)|^2}{|\xi|^{2\sigma+2s-2}}\,\mathrm{d}\xi\\
&\quad\ -\int_{|\xi|\in[\varepsilon_0,\varepsilon_1]}|\widehat{\ml{K}}(t,\xi;\sigma)|^2\,|\widehat{v}_0(\xi)|^2\,\mathrm{d}\xi,
\end{align*}	
where $0<\varepsilon_0<\varepsilon_1$ such that $|\xi|^{\sigma}t\leqslant\varepsilon_1^{\sigma}t_0<1$ for $t=t_0$. For this reason, it holds that
\begin{align*}
\frac{|\sin(|\xi|^{\sigma}t_0)|^2}{|\xi|^{2\sigma}}\geqslant\frac{t_0^2}{4} \ \ \mbox{when}\ \ |\xi|\in[\varepsilon_0,\varepsilon_1],
\end{align*}
and
\begin{align*}
\ml{I}_{n,\sigma,s}(t_0)&\gtrsim t_0^2|P_{v_1}|^2\int_{|\xi|\in[\varepsilon_0,\varepsilon_1]}|\xi|^{-2s}\,\mathrm{d}\xi-t_0^2\|v_1\|_{L^{1,1}}^2\int_{|\xi|\in[\varepsilon_0,\varepsilon_1]}|\xi|^{2-2s}\,\mathrm{d}\xi-\|v_0\|_{L^2}^2\\
&\gtrsim t_0^2|P_{v_1}|^2\int_{\varepsilon_0}^{\varepsilon_1}r^{n-1-2s}\,\mathrm{d}r-t_0^2\|v_1\|_{L^{1,1}}^2\int_{\varepsilon_0}^{\varepsilon_1}r^{n+1-2s}\,\mathrm{d}r-\|v_0\|_{L^2}^2,
\end{align*}
where the unexpressed multiplicative constants are independent of $t_0,\varepsilon_0,\varepsilon_1$. Some direct computations imply that
\begin{align*}
\ml{I}_{n,\sigma,s}(t_0)\gtrsim\begin{cases}
\displaystyle{t_0^2|P_{v_1}|^2\left(\ln\varepsilon_0^{-1}+\ln\varepsilon_1\right)-\frac{t_0^2\|v_1\|_{L^{1,1}}^2}{2}\left(\varepsilon_1^2-\varepsilon_0^2\right)-\|v_0\|_{L^2}^2}\ \ \mbox{if}\ \ n=2s,\\[0.7em]
\displaystyle{t_0^2|P_{v_1}|^2\,\frac{(\varepsilon_0^{-1})^{2s-n}-\varepsilon_1^{n-2s}}{2s-n}-\|v_0\|_{L^2}^2}&\\[0.7em]
\qquad\qquad\qquad\ -t_0^2\|v_1\|_{L^{1,1}}^2\times\begin{cases}
\displaystyle{\frac{\varepsilon_1^{n-2s+2}-\varepsilon_0^{n-(2s-2)}}{n-(2s-2)}}&\mbox{if}\ \ n\in(2s-2,2s),\\[0.7em]
\displaystyle{\ln\varepsilon_0^{-1}+\ln\varepsilon_1}&\mbox{if}\ \ n=2s-2,\\[0.7em]
\displaystyle{\frac{(\varepsilon_0^{-1})^{(2s-2)-n}-\varepsilon_1^{n-2s+2}}{(2s-2)-n}}&\mbox{if}\ \ n\in[1,2s-2).
\end{cases}
\end{cases}
\end{align*}
Letting $\varepsilon_0\downarrow 0$, i.e. $\varepsilon_0^{-1}\to+\infty$, but $\varepsilon_1$ is fixed, it immediately shows that $\ml{I}_{n,\sigma,s}(t_0)=+\infty$ at any finite time $t=t_0>0$, because $|P_{v_1}|>0$ and $(\varepsilon_0^{-1})^{2s-n}\gg \ln\varepsilon_0^{-1},(\varepsilon_0^{-1})^{(2s-2)-n}$ when $2s>n$. The proof is complete.
\end{proof}

Finally, we take $v_1\in L^2$ only, i.e. without any additional $L^1$ integrability, to investigate its large time behavior. Considering a special class of initial data, we just can expect polynomially growth estimates for $\ml{I}_{n,\sigma,0}(t)$. In other words, in comparison with Propositions \ref{Prop-Upper} and \ref{Prop-Lower}, the additional $L^1$ (or, further $L^{1,1}$) integrability for $v_1$ seems essential to stabilize the quantity $\ml{I}_{n,\sigma,0}(t)$ for higher dimensions, which does not contain any stronger singularity, as large time $t\gg1$.
\begin{prop}\label{Prop-L2-data}
Let $\sigma\geqslant 1$, and $|\widehat{\ml{K}}(t,\xi;\sigma)|\lesssim 1$ uniformly in $t,\xi$. Let us assume $v_0\in L^2$ and $v_1\in L^2$ such that $|\widehat{v}_1(\xi)|\gtrsim|\xi|^{-\frac{n}{2}+\frac{\sigma\epsilon}{2}}$ for $|\xi|\in(0,1]$ with a sufficiently small constant $\epsilon>0$. Then, the time-dependent function $\ml{I}_{n,\sigma,0}(t)$ satisfies the following almost optimal estimates:
\begin{align*}
t^{2-\epsilon}\lesssim\ml{I}_{n,\sigma,0}(t)\lesssim t^2\|(v_0,v_1)\|_{L^2\times L^2}^2
\end{align*}
for large time $t\gg1$.
\end{prop}
\begin{remark}
Under the condition $v_1\in L^2$, the function $\widehat{v}_1$ is well-defined in $L^2$ also. Our assumption for $\widehat{v}_1$ is localized for small frequencies only, which guarantees the almost sharpness of lower bound. It is worth noting that $v_1\not\in L^2$ if $\epsilon=0$ according to \eqref{L2data}.
\end{remark}
\begin{proof}
The upper bound estimate is trivial from
\begin{align*}
\ml{I}_{n,\sigma,0}(t)&\lesssim\int_{\mb{R}^n}|\widehat{\ml{K}}(t,\xi;\sigma)|^2\,|\widehat{v}_0(\xi)|^2\,\mathrm{d}\xi+t^2\int_{\mb{R}^n}\left|\frac{\sin(|\xi|^{\sigma}t)}{|\xi|^{\sigma}t}\right|^2|\widehat{v}_1(\xi)|^2\,\mathrm{d}\xi\\
&\lesssim \|v_0\|_{L^2}^2+t^2\|v_1\|_{L^2}^2.
\end{align*}
Concerning its lower bound we know that for large time such that $\nu_0=(\frac{\pi}{4})^{\frac{1}{\sigma}}t^{-\frac{1}{\sigma}}< 1$ the domain of integration can be shrunk by $[\nu_0,1]\cap(\cup_{j=0}^{+\infty}[\nu_j,\mu_j])$, to be specific,
\begin{align*}
\int_{\mb{R}^n}\frac{|\sin(|\xi|^{\sigma}t)|^2}{|\xi|^{2\sigma}}\,|\widehat{v}_1(\xi)|^2\,\mathrm{d}\xi&\geqslant\int_{|\xi|\in[\nu_0,1]}\frac{|\sin(|\xi|^{\sigma}t)|^2}{|\xi|^{2\sigma}}\,|\widehat{v}_1(\xi)|^2\,\mathrm{d}\xi\\
&\geqslant\sum\limits_{j=0}^{+\infty}\int_{|\xi|\in[\nu_0,1]\cap[\nu_j,\mu_j]}\left(\frac{1}{\sqrt{2}}\right)^2|\xi|^{-2\sigma}\,|\widehat{v}_1(\xi)|^2\,\mathrm{d}\xi,
\end{align*}
where we used \eqref{Treatment-sine}.
From our assumption $|\widehat{v}_1(\xi)|^2\gtrsim |\xi|^{-n+\sigma\epsilon}$ for $|\xi|\in(0,1]$ as well as $\nu_0>0$, we may further obtain
\begin{align*}
\int_{\mb{R}^n}\frac{|\sin(|\xi|^{\sigma}t)|^2}{|\xi|^{2\sigma}}\,|\widehat{v}_1(\xi)|^2\,\mathrm{d}\xi&\gtrsim \sum\limits_{j=0}^{+\infty}\int_{|\xi|\in[\nu_0,1]\cap[\nu_j,\mu_j]}|\xi|^{-2\sigma-n+\sigma\epsilon}\,\mathrm{d}\xi\\
&\gtrsim \sum\limits_{j=0}^{+\infty}\int_{r\in[\nu_0,1]\cap[\nu_j,\mu_j]}r^{-2\sigma-1+\sigma\epsilon}\,\mathrm{d}r\\
&\gtrsim\frac{1}{2}\int_{(\frac{\pi}{4})^{\frac{1}{\sigma}}t^{-\frac{1}{\sigma}}}^1r^{-2\sigma-1+\sigma\epsilon}\,\mathrm{d}r\gtrsim t^{2-\epsilon}
\end{align*}
for large time $t\gg1$, where we employed $-2\sigma-1+\sigma\epsilon<0$ with $0<\epsilon\ll\sigma$, i.e. $r^{-2\sigma-1+\sigma\epsilon}$ is a monotone decreasing function, and the length property in \eqref{Seq-02}.
Particularly, this initial data belongs to $L^2$ only since
\begin{align}\label{L2data}
\int_{|\xi|\leqslant 1}|\widehat{v}_1(\xi)|^2\,\mathrm{d}\xi\gtrsim\int_0^1r^{-1+\sigma\epsilon}\,\mathrm{d}r=\frac{1}{\sigma\epsilon} \ \ \mbox{and}\ \ v_1\in L^2.
\end{align}
The proof is finished.
\end{proof}

\subsection{Proofs of Theorems \ref{Thm-L1-data} and \ref{Thm-L2-data}}\label{Sub-Sec-Proof}
\hspace{5mm}According to the standard Fourier transform with respect to the spatial variable to the free wave equation \eqref{Eq-Waves}, the unknown $\widehat{u}=\widehat{u}(t,\xi)$ satisfies the following Cauchy problem (see some PDEs textbooks, e.g. \cite[Chapter 14]{Ebert-Reissig=2018}):
\begin{align*}
\begin{cases}
\widehat{u}_{tt}+|\xi|^2\,\widehat{u}=0,&\xi\in\mb{R}^n,\ t>0,\\
(\widehat{u},\widehat{u}_t)(0,\xi)=(\widehat{u}_0,\widehat{u}_1)(\xi),&\xi\in\mb{R}^n,
\end{cases}
\end{align*}
whose solution is uniquely addressed by
\begin{align*}
\widehat{u}(t,\xi)=\cos(|\xi|t)\,\widehat{u}_0(\xi)+\frac{\sin(|\xi|t)}{|\xi|}\,\widehat{u}_1(\xi).
\end{align*}
That is to say, we can apply Propositions \ref{Prop-Upper}-\ref{Prop-L2-data} with $\widehat{\ml{K}}(t,\xi;1)=\cos(|\xi|t)$, $\widehat{v}_0(\xi)=\widehat{u}_0(\xi)$, $\widehat{v}_1(\xi)=|\xi|^s\widehat{u}_1(\xi)$ and $\sigma=1$ to complete our proofs of Theorems \ref{Thm-L1-data}-\ref{Thm-L2-data}.
 
\subsection{Proof of Theorem \ref{Thm-stable}}\label{Sub-Sec-Proof-2}
\hspace{5mm}Let us turn back to the general kernel $\ml{I}_{n,\sigma,s}(t)$ firstly with a stronger assumption on $v_1$ to stabilize it globally in time, where $|\widehat{\ml{K}}(t,\xi;\sigma)|\lesssim 1$ still holds uniformly in $t,\xi$. In Propositions \ref{Prop-Upper}-\ref{Prop-Blow-up}, we have demonstrated
\begin{align*}
\ml{I}_{n,\sigma,s}(t)\approx\begin{cases}
+\infty&\mbox{if}\ \ n\in[1,2s],\\
t^{2-\frac{n-2s}{\sigma}}&\mbox{if}\ \ n\in(2s,2\sigma+2s),\\
\ln t&\mbox{if}\ \ n=2\sigma+2s,
\end{cases}
\end{align*}
which is not stable even for local in time (when $n\leqslant 2s$), provided that $v_0\in L^2$ and $v_1\in L^2\cap L^{1,1}$ with $|P_{v_1}|\neq0$. To improve the last optimal estimates, one may try to use $|P_{v_1}|=0$. For this reason, we in the following discussion only focus on the case $n\in[1, 2\sigma+2s]$, where in the opposite case $n\in(2\sigma+2s,+\infty)$ the quantity $\ml{I}_{n,\sigma,s}(t)$ is stable for large time with the initial data $v_1\in L^2\cap L^{1,1}$ and $|P_{v_1}|\neq0$.

\begin{prop}\label{Prop-Stable}
	Let $\sigma\geqslant 1$, and $|\widehat{\ml{K}}(t,\xi;\sigma)|\lesssim 1$ uniformly in $t,\xi$. Let us assume $v_0\in L^2$ and $v_1\in L^2\cap L^{1,\kappa}$ with $\lfloor\kappa\rfloor>\sigma+s-\frac{n}{2}+1$ in the lower dimensional case $n\in[1, 2\sigma+2s]$ such that
	\begin{align}\label{assumption-v1}
		\int_{\mb{R}^n}x^{\alpha}\,v_1(x)\,\mathrm{d}x\begin{cases}
			=0&\mbox{if}\ \ |\alpha|<\lfloor\kappa\rfloor-1,\\
			\neq0&\mbox{if}\ \ |\alpha|=\lfloor\kappa\rfloor-1.
		\end{cases}
	\end{align} Then, the time-dependent function $\ml{I}_{n,\sigma,s}(t)$ satisfies the following optimal estimate:
 \begin{align*}
	|P_{x^{\lfloor\kappa\rfloor-1}v_1}|^2\lesssim \ml{I}_{n,\sigma,s}(t)\lesssim \|v_0\|_{L^2}^2+\|v_1\|_{L^2\cap L^{1,\kappa}}^2
\end{align*}
	for large time $t\gg1$.
\end{prop}
\begin{proof}
Recalling the notation
\begin{align*}
	\ml{M}_{\alpha}(f):=\frac{(-1)^{|\alpha|}}{\alpha!}\int_{\mb{R}^n}x^{\alpha}f(x)\,\mathrm{d}x\ \ \mbox{with}\ \ |\alpha|\leqslant\lfloor\kappa\rfloor 
\end{align*}
for $f\in L^{1,\kappa}$, the recent work \cite[Lemma 5.1]{Ikehata-Michihisa=2019} has proved that
\begin{align*}
	\left|\widehat{f}(\xi)-\sum\limits_{|\alpha|\leqslant \lfloor \kappa \rfloor}\ml{M}_{\alpha}(f)(i\xi)^{\alpha}\right|\leqslant C_{\kappa}\,|\xi|^{\kappa}\int_{\mb{R}^n}|x|^{\kappa}|f(x)|\,\mathrm{d}x 
\end{align*}
with $n\geqslant 1$ and $\kappa\geqslant0$, where $C_{\kappa}>0$ is a constant independent of $\xi$ and $f$. 

According to our assumption \eqref{assumption-v1} on the initial data $v_1$ and the last inequality, we may conclude $\ml{M}_{\alpha}(v_1)=0$ for any $|\alpha|<\lfloor\kappa\rfloor-1$, furthermore,
\begin{align*}
	\left|\widehat{v}_1(\xi)-[\ml{M}_{\alpha}(v_1)(i\xi)^{\alpha}]\big|_{|\alpha|=\lfloor\kappa\rfloor-1}\right|&\leqslant[\ml{M}_{\alpha}(v_1)(i\xi)^{\alpha}]\big|_{|\alpha|=\lfloor\kappa\rfloor}+ C_{\kappa}\,|\xi|^{\kappa}\int_{\mb{R}^n}|x|^{\kappa}|v_1(x)|\,\mathrm{d}x\\
	&\lesssim |\xi|^{\lfloor\kappa\rfloor}\|v_1\|_{L^{1,\lfloor\kappa\rfloor}}+|\xi|^{\kappa}\|v_1\|_{L^{1,\kappa}}\\
	&\lesssim (|\xi|^{\lfloor\kappa\rfloor}+|\xi|^{\kappa})\|v_1\|_{L^{1,\kappa}}
\end{align*} 
and, by the triangle inequality,
\begin{align}\label{Est-10}
	|\widehat{v}_1(\xi)|&\lesssim|\xi|^{\lfloor\kappa\rfloor-1}\|v_1\|_{L^{1,\lfloor\kappa\rfloor-1}}+(|\xi|^{\lfloor\kappa\rfloor}+|\xi|^{\kappa})\|v_1\|_{L^{1,\kappa}}\notag\\
	&\lesssim (|\xi|^{\lfloor\kappa\rfloor-1}+|\xi|^{\lfloor\kappa\rfloor}+|\xi|^{\kappa})\|v_1\|_{L^{1,\kappa}},
\end{align}
where we used $L^{1,\lfloor\kappa\rfloor-1}\supset L^{1,\lfloor\kappa\rfloor}\supseteq L^{1,\kappa}$. Note that $\kappa\geqslant 2$ due to $\lfloor \kappa\rfloor>\sigma+s-\frac{n}{2}+1$ and $\sigma+s-\frac{n}{2}\geqslant0$.
As a consequence, the large time upper bound can be derived by
\begin{align}\label{Bdd-est}
	\ml{I}_{n,\sigma,s}(t)&\lesssim \|v_0\|_{L^2}^2+\left\|\sin(|\xi|^{\sigma}t)(|\xi|^{\lfloor\kappa\rfloor-1}+|\xi|^{\lfloor\kappa\rfloor}+|\xi|^{\kappa})|\xi|^{-\sigma-s}\|v_1\|_{L^{1,\kappa}}\right\|_{L^2(|\xi|\leqslant 1)}^2+\|\widehat{v}_1\|_{L^2(|\xi|\geqslant 1)}^2\notag\\
	&\lesssim \|v_0\|_{L^2}^2+\|v_1\|_{L^2}^2+\int_0^1r^{n-1+2\lfloor\kappa\rfloor-2-2\sigma-2s}\,\mathrm{d}r\,\|v_1\|_{L^{1,\kappa}}^2\notag\\
	&\lesssim \|v_0\|_{L^2}^2+\|v_1\|_{L^2\cap L^{1,\kappa}}^2,
\end{align}
in which we used $n-1+2\lfloor\kappa\rfloor-2-2\sigma-2s>-1$, i.e. $\lfloor\kappa\rfloor>\sigma+s-\frac{n}{2}+1$. This estimate already showed the desired stability.
\begin{remark}\
	By assuming $\int_{\mb{R}^n}x^{\alpha}\,v_1(x)\,\mathrm{d}x=0$ if $|\alpha|\leqslant\lfloor\kappa\rfloor$ strongly in \eqref{assumption-v1}, then the right-hand side of estimate \eqref{Est-10} can be improved by $|\xi|^{\kappa}\|v_1\|_{L^{1,\kappa}}$. It justifies the boundedness in \eqref{Bdd-est} with the aid of
	\begin{align*}
		\int_0^1 r^{n-1+2\kappa-2\sigma-2s}\,\mathrm{d}r<+\infty
	\end{align*}
	when $\kappa>\sigma+s-\frac{n}{2}$ for $v_1\in L^2\cap L^{1,\kappa}$. The special cases with $(\sigma,s)=(2,0)$ when $n=1$ with $\kappa\in(\frac{3}{2},2]$, and $n=2$ with $\kappa\in(1,2]$, exactly coincide with those of \cite[Theorem 1.4]{Ikehata=2024}; when $n=3$ with $\kappa\in(\frac{1}{2},1]$, and $n=4$ with $\kappa\in(0,1]$, exactly coincide with those of \cite[Theorem 1.3]{Ikehata=2024}.
\end{remark}

Concerning its optimal lower bound for large time, thanks to the additional assumption \eqref{assumption-v1}, i.e. $[\ml{M}_{\alpha}(v_1)]|_{|\alpha|=\lfloor\kappa\rfloor-1}\neq0$, we apply the triangle inequality to give
\begin{align}
	|\widehat{v}_1(\xi)|^2\geqslant\frac{1}{2[(\lfloor\kappa\rfloor-1)!]^2}\,|\xi|^{2(\lfloor\kappa\rfloor-1)}|P_{x^{\lfloor\kappa\rfloor-1}v_1}|^2-\widetilde{M}^2(|\xi|^{2\lfloor\kappa\rfloor}+|\xi|^{2\kappa})\|v_1\|_{L^{1,\kappa}}^2\label{Est-06}
\end{align}
with a positive constant $\widetilde{M}$. 
Therefore, the estimate of $I_n^{(3)}(t;\gamma)$ in \eqref{Est-01} is changed into
\begin{align*}
	I_n^{(3)}(t;\gamma)&\geqslant\frac{|P_{x^{\lfloor\kappa\rfloor-1}v_1}|^2}{4[(\lfloor\kappa\rfloor-1)!]^2}\int_{\mb{R}^n}\mathrm{e}^{-\gamma|\xi|^2}\frac{|\sin(|\xi|^{\sigma}t)|^2}{|\xi|^{2\sigma+2s-2\lfloor\kappa\rfloor+2}}\,\mathrm{d}\xi\\
	&\quad\ -\frac{\widetilde{M}^2}{2}\|v_1\|_{L^{1,\kappa}}^2\int_{\mb{R}^n}\mathrm{e}^{-\gamma|\xi|^2}\frac{|\sin(|\xi|^{\sigma}t)|^2}{|\xi|^{2\sigma+2s}}\,(|\xi|^{2\lfloor\kappa\rfloor}+|\xi|^{2\kappa})\,\mathrm{d}\xi.
\end{align*}
By following the same procedure as \eqref{Est-02}, e.g. the use of mapping $s\mapsto s-\lfloor\kappa\rfloor+1$ in it because of $n>2\sigma+2(s-\lfloor\kappa\rfloor+1)$, 
the dominant term is estimated by
\begin{align*}
	\frac{|P_{x^{\lfloor\kappa\rfloor-1}v_1}|^2}{4[(\lfloor\kappa\rfloor-1)!]^2}\int_{\mb{R}^n}\mathrm{e}^{-\gamma|\xi|^2}\frac{|\sin(|\xi|^{\sigma}t)|^2}{|\xi|^{2\sigma+2s-2\lfloor\kappa\rfloor+2}}\,\mathrm{d}\xi\gtrsim \gamma^{-\frac{n-2\sigma-2(s-\lfloor\kappa\rfloor+1)}{2}}|P_{x^{\lfloor\kappa\rfloor-1}v_1}|^2
\end{align*}
for large time $t\gg1$. Similarly, the change of variable hints
\begin{align*}
	&\frac{\widetilde{M}^2}{2}\|v_1\|_{L^{1,\kappa}}^2\int_{\mb{R}^n}\mathrm{e}^{-\gamma|\xi|^2}\frac{|\sin(|\xi|^{\sigma}t)|^2}{|\xi|^{2\sigma+2s}}\,(|\xi|^{2\lfloor\kappa\rfloor}+|\xi|^{2\kappa})\,\mathrm{d}\xi\\
	&\qquad\lesssim\left(\gamma^{-\frac{n+2-2\sigma-2(s-\lfloor\kappa\rfloor+1)}{2}}+\gamma^{-\frac{n+2-2\sigma-2(s-\kappa+1)}{2}}\right)\|v_1\|_{L^{1,\kappa}}^2\\
	&\qquad\lesssim \gamma^{-\frac{n-2\sigma-2(s-\lfloor\kappa\rfloor+1)}{2}-1}\|v_1\|_{L^{1,\kappa}}^2
\end{align*}
for suitable large $\gamma>1$. Thus, by choosing a suitable large (but fixed) constant $\gamma$ similarly to the one of \eqref{Est-03}, we directly conclude
\begin{align*}
	I_n^{(3)}(t;\gamma)\gtrsim \frac{1}{2}\,\gamma^{-\frac{n-2\sigma-2(s-\lfloor\kappa\rfloor+1)}{2}}|P_{x^{\lfloor\kappa\rfloor-1}v_1}|^2\ \ \Rightarrow\ \  \ml{I}_{n,\sigma,s}(t)\gtrsim|P_{x^{\lfloor\kappa\rfloor-1}v_1}|^2,
\end{align*}
which ensures the large time optimality of bounded estimate \eqref{Bdd-est} for $v_0\in L^2$ and $v_1\in L^{1,\kappa}$ with \eqref{assumption-v1} if $\lfloor\kappa\rfloor>\sigma+s-\frac{n}{2}+1$ when $n\in[1, 2\sigma+2s]$. It completes our proof.
\end{proof}

 \begin{remark}\label{Rem-critical}
 Let us consider a class of limit regularity case $\lfloor\kappa\rfloor-1=\sigma+s-\frac{n}{2}\in\mb{N}_0$ for some $(n,\sigma,s)$. Then, by using the same techniques as \eqref{Est-05} associate with \eqref{Est-06}, recalling
 \begin{align*}
 	2\min\{\lfloor \kappa\rfloor,\kappa\}-2\sigma-2s+n-1=1>-1,
 \end{align*}
  one derives
 \begin{align*}
 \ml{I}_{n,\sigma,s}(t)&\geqslant\frac{1}{2}\int_{\mb{R}^n}\mathrm{e}^{-|\xi|^2}\,\frac{|\sin(|\xi|^{\sigma}t)|^2}{|\xi|^{2\sigma+2s}}\,|\widehat{v}_1(\xi)|^2\,\mathrm{d}\xi-\|v_0\|_{L^2}^2\\
 &\gtrsim|P_{x^{\lfloor\kappa\rfloor-1}v_1}|^2\int_{\mb{R}^n}\mathrm{e}^{-|\xi|^2}\,|\sin(|\xi|^{\sigma}t)|^2\,|\xi|^{2(\lfloor\kappa\rfloor-1)-2\sigma-2s}\,\mathrm{d}\xi\\
 &\quad\ -\|v_1\|_{L^{1,\kappa}}^2\int_{\mb{R}^n}\mathrm{e}^{-|\xi|^2}\,|\sin(|\xi|^{\sigma}t)|^2(|\xi|^{2\lfloor\kappa\rfloor}+|\xi|^{2\kappa})|\xi|^{-2\sigma-2s}\,\mathrm{d}\xi-\|v_0\|_{L^2}^2\\
 &\gtrsim|P_{x^{\lfloor\kappa\rfloor-1}v_1}|^2\int_{\nu_0}^{+\infty}\mathrm{e}^{-r^2}\,r^{-1}\,\mathrm{d}r-\|v_1\|_{L^{1,\kappa}}^2\int_0^{+\infty}\mathrm{e}^{-r^2}\,(r^{2\lfloor\kappa\rfloor}+r^{2\kappa})\,r^{-2\sigma-2s+n-1}\,\mathrm{d}r-\|v_0\|_{L^2}^2\\
 &\gtrsim \ln t\,|P_{x^{\lfloor\kappa\rfloor-1}v_1}|^2-\|v_1\|_{L^{1,\kappa}}^2-\|v_0\|_{L^2}^2,
 \end{align*}
which leads to $\ml{I}_{n,\sigma,s}(t)\gtrsim \ln t\,|P_{x^{\lfloor\kappa\rfloor-1}v_1}|^2$ for large time $t\gg1$. Note that
\begin{align*}
\int_{\mb{R}^n}\mathrm{e}^{-|\xi|^2}\,|\sin(|\xi|^{\sigma}t)|^2\,|\xi|^{2(\lfloor\kappa\rfloor-1)-2\sigma-2s}\,\mathrm{d}\xi&\geqslant\sum\limits_{j=0}^{+\infty}\int_{\nu_j\leqslant|\xi|\leqslant\mu_j}\mathrm{e}^{-|\xi|^2}\left(\frac{1}{\sqrt{2}}\right)^2|\xi|^{-n}\,\mathrm{d}\xi\\
&\geqslant\frac{\omega_n}{4}\int_{\nu_0}^{+\infty}\mathrm{e}^{-r^2}\,r^{-1}\,\mathrm{d}r,
\end{align*}
analogously to \eqref{Est-07}.
 That is to say, in such critical regularity case $\lfloor\kappa\rfloor=\sigma+s-\frac{n}{2}+1$, by assuming $|P_{x^{\lfloor\kappa\rfloor-1}v_1}|\neq0$, the time-dependent function $\ml{I}_{n,\sigma,s}(t)$ is still instable for large time. The special case with $(\kappa,n,\sigma,s)=(2,2,2,0)$ and $|P_{u_1}|=0$ exactly coincide with the one of \cite[Case (2) of Theorem 1.5]{Ikehata=2024} So, the  condition $\lfloor\kappa\rfloor>\sigma+s-\frac{n}{2}+1$ is necessary in general and sufficient for the global (in time) stability of $\ml{I}_{n,\sigma,s}(t)$ with $L^2\cap L^{1,\kappa}$ data associated with the assumption \eqref{assumption-v1}.
 \end{remark}

Finally, we just need to apply Proposition \ref{Prop-Stable} with $\widehat{\ml{K}}(t,\xi;1)=\cos(|\xi|t)$, $\widehat{v}_0(\xi)=\widehat{u}_0(\xi)$, $\widehat{v}_1(\xi)=|\xi|^s\widehat{u}_1(\xi)$ and $\sigma=1$, as usual, to complete our proof of Theorem \ref{Thm-stable}.

\section{Applications of our results to other evolution models}\setcounter{equation}{0}\label{Section-Applications}
\hspace{5mm}In this section, we will apply Propositions \ref{Prop-Upper}-\ref{Prop-Blow-up} concerning $L^2\cap L^{1,1}$ data for some evolution equations. Absolutely, one may also apply Proposition \ref{Prop-L2-data} concerning $L^2$ data, and Proposition \ref{Prop-Stable} concerning $L^2\cap L^{1,\kappa}$ data, for these evolution equations without any additional difficulty (we do not show all of them for the sake of briefness).
\subsection{Wave equation with scale-invariant terms}
\hspace{5mm}Let us consider the linear scale-invariant wave equation with dissipation and mass, e.g. \cite{D'Abbicco-Lucente-Reissig=2015,Nasci-Palmieri-Reissig=2017,Palmieri-Reissig=2019,Palmieri-Tu=2021}, namely,
\begin{align}\label{Eq-Scale-Invariant-Waves}
\begin{cases}
\displaystyle{\ml{u}_{tt}-\Delta \ml{u}+\frac{\tau_1}{1+t}\,\ml{u}_t+\frac{\tau_2}{(1+t)^2}\,\ml{u}=0,}&x\in\mb{R}^n,\ t>0,\\
(\ml{u},\ml{u}_t)(0,x)=(\ml{u}_0,\ml{u}_1)(x),&x\in\mb{R}^n,
\end{cases}
\end{align}
with the parameters $\tau_1>0$ and $\tau_2\in\mb{R}$. As the previous literature \cite{Nasci-Palmieri-Reissig=2017,Palmieri-Reissig=2019,Palmieri-Tu=2021} and references therein, it turns out that the quantity $\delta_{\tau_1,\tau_2}:=(\tau_1-1)^2-4\tau_2$ is useful to describe some of the properties to \eqref{Eq-Scale-Invariant-Waves}. We in this part focus on $\delta_{\tau_1,\tau_2}=1$, in other words, $\tau_2\equiv\frac{\tau_1^2}{4}-\frac{\tau_1}{2}$. By applying the dissipative transformation
\begin{align*}
\ml{u}(t,x)=(1+t)^{-\frac{\tau_1}{2}}u(t,x),
\end{align*}
since $\delta_{\tau_1,\tau_2}=1$, the Cauchy problem \eqref{Eq-Scale-Invariant-Waves} is transformed to the free wave equation \eqref{Eq-Waves} with the initial data $u_0(x)=\ml{u}_0(x)$ and $u_1(x)=\frac{\tau_1}{2}\ml{u}_0(x)+\ml{u}_1(x)$. Then, according to Theorem \ref{Thm-L1-data}, we may arrive at the optimal estimates
\begin{align*}
	t^{-\frac{\tau_1}{2}}\ml{D}_{n,1,s}(t)|P_{\frac{\tau_1}{2}\nabla^s \ml{u}_0+\nabla^s\ml{u}_1}|\lesssim\|\ml{u}(t,\cdot)\|_{L^2}\lesssim t^{-\frac{\tau_1}{2}}\ml{D}_{n,1,s}(t)\|(\ml{u}_0,\nabla^s\ml{u}_0,\nabla^s\ml{u}_1)\|_{L^2\times(L^2\cap L^1)^2}
\end{align*}
for large time $t\gg1$, provided that $|P_{\frac{\tau_1}{2}\nabla^s \ml{u}_0+\nabla^s\ml{u}_1}|\neq0$, when $n\in(2s,+\infty)$; and the blow-up in finite time $\|\ml{u}(t,\cdot)\|_{L^2}=+\infty$ for $t\in(0,t_0]$ when $n\in[1,2s]$. Especially, by choosing $\tau_1=2+2s-n$ in the lower dimensions $n\in(2s,2+2s)$, the lower order terms (positive dissipation term but negative mass term)
\begin{align*}
\frac{2+2s-n}{1+t}\,\ml{u}_t-\frac{(n-2s)(2+2s-n)}{4(1+t)^2}\,\ml{u}
\end{align*}
is sufficient to globally (in time) stabilize  the free wave equation with $\nabla^s(L^2\cap L^1)$ initial data such that $\|\ml{u}(t,\cdot)\|_{L^2}\approx |P_{\frac{\tau_1}{2}\nabla^s \ml{u}_0+\nabla^s\ml{u}_1}|$ for any $s\geqslant0$.

\subsection{Undamped $\sigma$-evolution equation}
\hspace{5mm}Let us consider the undamped $\sigma$-evolution equation, e.g. \cite{Ebert-Reissig=2018,Ebert-Lour=2019,Ikehata=2024}, as a general wave model with fractional Laplacian, namely,
\begin{align}\label{Eq-sigma-eq}
\begin{cases}
w_{tt}+(-\Delta)^{\sigma}w=0,&x\in\mb{R}^n,\ t>0,\\
(w,w_t)(0,x)=(w_0,w_1)(x),&x\in\mb{R}^n,
\end{cases}
\end{align}
with $\sigma\geqslant 1$, particularly, $\sigma=2$ appears in the classical beam/plate equation. By taking $\widehat{\ml{K}}(t,\xi;\sigma)=\cos(|\xi|^{\sigma}t)$, $\widehat{v}_0(\xi)=\widehat{w}_0(\xi)$, $\widehat{v}_1(\xi)=|\xi|^s\widehat{w}_1(\xi)$ in Propositions \ref{Prop-Upper}-\ref{Prop-Blow-up}, the solution to \eqref{Eq-sigma-eq} in the Fourier space is expressed via
\begin{align*}
\widehat{w}(t,\xi)=\widehat{\ml{K}}(t,\xi;\sigma)\,\widehat{v}_0(\xi)+\frac{\sin(|\xi|^{\sigma}t)}{|\xi|^{\sigma}}\,\widehat{v}_1(\xi).
\end{align*}
Then, we may arrive at the optimal estimates
\begin{align*}
\ml{D}_{n,\sigma,s}(t)|P_{\nabla^s w_1}|\lesssim\|w(t,\cdot)\|_{L^2}\lesssim\ml{D}_{n,\sigma,s}(t)\|(w_0,\nabla^sw_1)\|_{L^2\times(L^2\cap L^1)}
\end{align*}
for large time $t\gg1$, provided that $|P_{\nabla^s w_1}|\neq0$, when $n\in(2s,+\infty)$; and the blow-up in finite time $\|w(t,\cdot)\|_{L^2}=+\infty$ for $t\in(0,t_0]$ when $n\in[1,2s]$. These results coincide with \cite[Theorems 1.1 and 1.2]{Ikehata=2024}, where \cite{Ikehata=2024} considered $s=0$, $n\leqslant 4$ and $\sigma=2$. This result carrying $s=0$ gives a positive answer to the conjecture from \cite[Remark 1.1]{Ikehata=2024} via the critical dimension $n=2\sigma$. Furthermore, it answers the unknown general case $\sigma\geqslant 1$ even with higher regularity $s> 0$.

\subsection{Critical Moore-Gibson-Thompson equation}
\hspace{5mm}Let us consider the critical Moore-Gibson-Thompson (MGT) equation, e.g. \cite{Kaltenbacher-Lasiecka-Marchand=2011,Marchand-McDevitt-Triggiani=2012,Chen-Palmieri=2020,Chen=2024}, arising in  linear acoustics of inviscid fluids (i.e. the Lighthill approximation of compressible Euler-Cattaneo system under the irrotational flow), namely,
\begin{align}\label{Eq-MGT}
\begin{cases}
\tau\psi_{ttt}+\psi_{tt}-\Delta\psi-\tau\Delta\psi_t=0,&x\in\mb{R}^n,\ t>0,\\
(\psi,\psi_t,\psi_{tt})(0,x)=(\psi_0,\psi_1,\psi_2)(x),&x\in\mb{R}^n,
\end{cases}
\end{align}
with the thermal relaxation $\tau>0$ from the application of Cattaneo law of heat conduction and the acoustic velocity potential $\psi=\psi(t,x)\in\mb{R}$. It was proved by \cite{Kaltenbacher-Lasiecka-Marchand=2011,Marchand-McDevitt-Triggiani=2012} that the exponential stability of its corresponding semigroup is lost (due to the lack of viscous dissipation).
 Let us re-organize the representation of solution to \eqref{Eq-MGT} in the Fourier space deduced in \cite[Proof of Proposition 2]{Chen-Palmieri=2020}. By taking
\begin{align*}
\widehat{\ml{K}}(t,\xi;1)\,\widehat{v}_0(\xi)&=\left(\frac{\cos(|\xi|t)}{1+\tau^2|\xi|^2}+\frac{\tau|\xi|\sin(|\xi|t)}{1+\tau^2|\xi|^2}+\frac{\tau^2|\xi|^2}{2(1+\tau^2|\xi|^2)}\,\mathrm{e}^{-\frac{t}{\tau}}\right)\widehat{\psi}_0(\xi)\\
&\quad\ +\left(-\frac{\tau^3|\xi|\sin(|\xi|t)}{1+\tau^2|\xi|^2}-\frac{\tau^2\cos(|\xi|t)}{1+\tau^2|\xi|^2}+\frac{\tau^2}{2(1+\tau^2|\xi|^2)}\,\mathrm{e}^{-\frac{t}{\tau}}\right)\widehat{\psi}_2(\xi),
\end{align*}
and $\widehat{v}_1(\xi)=|\xi|^s(\widehat{\psi}_1(\xi)+\tau\widehat{\psi}_2(\xi))$, $\sigma=1$ in Propositions \ref{Prop-Upper}-\ref{Prop-Blow-up}, the solution to \eqref{Eq-MGT} in the Fourier space is expressed via
\begin{align*}
\widehat{\psi}(t,\xi)=\widehat{\ml{K}}(t,\xi;1)\,\widehat{v}_0(\xi)+\frac{\sin(|\xi|t)}{|\xi|}\,\widehat{v}_1(\xi).
\end{align*}
Obviously, it holds
\begin{align*}
\frac{|\xi|+|\xi|^2}{1+\tau^2|\xi|^2}\lesssim 1\ \ \Rightarrow \ \ |\widehat{\ml{K}}(t,\xi;1)\,\widehat{v}_0(\xi)|\lesssim |\widehat{\psi}_0(\xi)|+|\widehat{\psi}_2(\xi)|.
\end{align*}
Then, we may arrive at the optimal estimates
\begin{align*}
\ml{D}_{n,1,s}(t)|P_{\nabla^s(\psi_1+\tau\psi_2)}|\lesssim \|\psi(t,\cdot)\|_{L^2}\lesssim\ml{D}_{n,1,s}(t)\|(\psi_0,\nabla^s\psi_1,\nabla^s\psi_2)\|_{L^2\times (L^2\cap L^1)\times (L^2\cap L^1)}
\end{align*}
for large time $t\gg1$, provided that $|P_{\nabla^s(\psi_1+\tau\psi_2)}|\neq0$, when $n\in(2s,+\infty)$; and the blow-up in finite time $\|\psi(t,\cdot)\|_{L^2}=+\infty$ for $t\in(0,t_0]$ when $n\in[1,2s]$. These results coincide with \cite[Theorem 4]{Chen=2024} when $s=0$ and $n\leqslant 2$, furthermore, improve that for the general regularity $s\geqslant 0$ and $n\geqslant 3$.

\subsection{Linearized compressible Euler system}
\hspace{5mm}Let us consider the well-known linearized compressible Euler system (see, for example, the textbook \cite{Pierce=1989}), describing the ideal gases from fluid mechanics, namely,
\begin{align}\label{Eq-Euler}
\begin{cases}
\rho_t+\beta\divv\mathbf{u}=0,&x\in\mb{R}^n,\ t>0,\\
\mathbf{u}_t+\beta\nabla\rho=0,&x\in\mb{R}^n,\ t>0,\\
(\rho,\mathbf{u})(0,x)=(\rho_0,\mathbf{u}_0)(x),&x\in\mb{R}^n,
\end{cases}
\end{align}
with the density $\rho=\rho(t,x)\in\mb{R}$ and the vector velocity $\mathbf{u}=\mathbf{u}(t,x)\in\mb{R}^n$, where $\beta$ is a positive constant. By straightforward computation, the solutions to \eqref{Eq-Euler} in the Fourier space are given by
\begin{align*}
\widehat{\rho}(t,\xi)&=\cos(\beta|\xi|t)\,\widehat{\rho}_0(\xi)-i\beta\frac{\sin(\beta|\xi|t)}{\beta|\xi|}\,\big(\xi\cdot\mathbf{u}_0(\xi)\big),\\
\widehat{\mathbf{u}}(t,\xi)&=\left(\mathbf{u}_0(\xi)-\frac{\xi(\xi\cdot\mathbf{u}_0(\xi))}{|\xi|^2}+\cos(\beta|\xi|t)\frac{\xi(\xi\cdot\mathbf{u}_0(\xi))}{|\xi|^2}\right)-i\beta\frac{\sin(\beta|\xi|t)}{\beta|\xi|}\,\xi\widehat{\rho}_0(\xi).
\end{align*}
According to Propositions \ref{Prop-Upper}-\ref{Prop-Blow-up} with suitable kernels $\widehat{\ml{K}}(t,\xi;1)$ and initial data, we may arrive at the optimal estimates
\begin{align*}
\ml{D}_{n,1,s}(t)|P_{\nabla^s\divv\mathbf{u}_0}|&\lesssim\|\rho(t,\cdot)\|_{L^2}\lesssim \ml{D}_{n,1,s}(t)\|(\rho_0,\nabla^s\divv\mathbf{u}_0)\|_{L^2\times(L^2\cap L^1)},\\
\ml{D}_{n,1,s}(t)|P_{\nabla^{s+1}\rho_0}|&\lesssim\|\mathbf{u}(t,\cdot)\|_{L^2}\lesssim \ml{D}_{n,1,s}(t)\|(\nabla^{s+1}\rho_0,\mathbf{u}_0)\|_{(L^2\cap L^1)\times L^2},
\end{align*}
for large time $t\gg1$, provided that $|P_{\nabla^s\divv\mathbf{u}_0}|\neq0\neq|P_{\nabla^{s+1}\rho_0}|$, when $n\in(2s,+\infty)$; and the blow-up in finite time $\|\rho(t,\cdot)\|_{L^2}=+\infty$, $\|\mathbf{u}(t,\cdot)\|_{L^2}=+\infty$ for $t\in(0,t_0]$ when $n\in[1,2s]$.

\section*{Acknowledgments} 
 Wenhui Chen is supported in part by the National Natural Science Foundation of China (grant No. 12301270, grant No. 12171317), Guangdong Basic and Applied Basic Research Foundation (grant No. 2025A, grant No. 2023A1515012044), 2024 Basic and Applied Basic Research Topic--Young Doctor Set Sail Project (grant No. 2024A04J0016). Ryo Ikehata is supported in part by Grant-in-Aid for scientific Research (C) 20K03682 of JSPS.

\end{document}